\documentclass[preprint]{elsarticle} 
\usepackage{amssymb,latexsym,amsmath,amscd,graphicx,setspace,amsthm,verbatim,float}
\usepackage[margin = 3 cm]{geometry}
\usepackage{mathtools}
\usepackage{tikz}
\usetikzlibrary{chains}

\tikzset{node distance=2em, ch/.style={circle,draw,on chain,inner sep=2pt},chj/.style={ch,join},every path/.style={shorten >=4pt,shorten <=4pt},line width=1pt,baseline=-1ex}

\let\dlabel=\alabel

\newcommand{\dnode}[2][chj]{%
\node[#1,label={below:\dlabel{#2}}] {};
}

\newcommand{\dnodenj}[1]{%
\dnode[ch]{#1}
}

\newcommand{\dnodebr}[1]{%
\node[chj,label={below right:\dlabel{#1}}] {};
}

\newcommand{\dydots}{%
\node[chj,draw=none,inner sep=1pt] {\dots};
}

\input xy
\xyoption{all}

\theoremstyle{plain}

\newtheorem{The}{Theorem}[section]
\newtheorem{Pro}[The]{Proposition}
\newtheorem{Cor}[The]{Corollary}
\newtheorem{Lem}[The]{Lemma}

\theoremstyle{remark}

\theoremstyle{definition}
\newtheorem{Def}[The]{Definition}
\def\cF{\mathcal{F}}
\def\cW{\mathcal{W}}
\def\cS{\mathcal{S}}
\def\cP{\mathcal{P}}
\def\Z{\mathbb{Z}}

\begin{document} 

\title{Weyl groups of some hyperbolic Kac-Moody algebras}

\author[ajf]{Alex J. Feingold\corref{cor1}}
    \ead{alex@math.binghamton.edu}
\author[dv]{Daniel Valli\`{e}res}
    \ead{dvallieres@csuchico.edu}
\cortext[cor1]{Corresponding author} 
\address[ajf]{Department of Mathematical Sciences, Binghamton University, State University of New York, 
Binghamton, New York 13902-6000}
\address[dv]{Department of Mathematics and Statistics, California State University - Chico, 400 West First Street, Chico, California 95929}

%\date{\today} 

\begin{abstract} 
We use the theory of Clifford algebras and Vahlen groups to study Weyl groups of hyperbolic Kac-Moody algebras $T_{n}^{++}$, obtained by a process of double extension from a Cartan matrix of finite type $T_{n}$, whose corresponding generalized Cartan matrices are symmetric.
\end{abstract} 
\maketitle 
\tableofcontents

\begin{section}{Introduction}

In \cite{FeFr83}, Feingold and Frenkel gained significant new insight into the structure of a particularly interesting rank $3$ hyperbolic Kac--Moody algebra which they called $\cF$ (also known as $A_{1}^{++}$), along with some connections to the theory of Siegel modular forms of genus $2$.  The first vital step in their work was the discovery that the even part of the Weyl group of that Kac-Moody algebra is $\cS\cW(\cF) \cong PSL(2,\Z)$.  (If $\cW$ is a Weyl group, we will denote its even part by $\cS\cW$.)

In \cite{FeKlNi09}, a coherent picture of Weyl groups was presented for many higher rank hyperbolic Kac--Moody algebras using lattices and subrings of the four normed division algebras. Specifically, the Weyl groups of all hyperbolic algebras of ranks $4$, $6$ and $10$ which can be obtained by a process of double extension, admit realizations in terms of generalized modular groups over the complex numbers $\mathbb{C}$, the quaternions $\mathbb{H}$, and the octonions $\mathbb{O}$, respectively.  In particular, the authors found in the rank four situation that the even part of the Weyl groups of the Kac-Moody algebra $A_{2}^{++}$ is the Bianchi group $PSL(2,O_{-3})$, where $O_{-3}$ the ring of integers of $\mathbb{Q}(\sqrt{-3})$. 

One could ask if there is a similar phenomenon for all the hyperbolic Kac-Moody algebras $T_{n}^{++}$, where $T_{n}$ is any finite type root system, but it is not clear what to take instead of a normed division algebra.  In \cite{KlNiPa}, the authors used the quaternions and the octonions in their study of some Weyl groups.  In this paper, we adopt another approach.  We use the theory of Vahlen groups and Clifford algebras in order to study the Weyl groups of the hyperbolic Kac-Moody algebras $T_{n}^{++}$ whose Cartan matrices are symmetric. A key ingredient needed to obtain our main result in Theorem \ref{main} is 
Corollary $5.10$ of \cite{Ka90}, which only applies to that class of Cartan matrices. We plan to study more general cases
where the Cartan matrices are Lorentzian, and believe that our methods will yield interesting results with connections to 
number theory.

Our paper is organized as follows.  In \S \ref{generalities}, we remind the reader about generalities on orthogonal geometries.  Then, we gather some results on Clifford algebras, Pin and Spin groups in \S \ref{clifpinspin}.  Section \S \ref{Vahlen1} introduces Vahlen groups.  In the literature, Vahlen groups have been defined for the paravector case as well as for the non-paravector case.  In \S \ref{Vahlen1}, we place ourselves in the non-paravector case, whereas the paravector situation is treated in \S \ref{paravector} and \S \ref{vahlenpara}.  Section \S \ref{cartanmatrices} contains a useful, though very simple, introduction to generalized Cartan matrices, systems of simple roots and Weyl groups.  The core of this paper in contained in \S \ref{symmetric}, where we give a description of several Weyl groups.  At last, we explain in \S \ref{connections} the connections between our approach and the one adopted previously in \cite{FeKlNi09}.

\vskip10pt
\noindent
{\bf Acknowledgements:}
We would like to thank Igor Frenkel for suggesting this direction of research. AJF gratefully acknowledges the hospitality 
of the Albert Einstein Institute on various visits, and the IH\'ES.  We would also like to thank Joel Dodge for useful 
discussions at the beginning of this project, and we thank the referee for interesting suggestions for further research.

\end{section}

%\section{Generalities on orthogonal geometries}
\section{Generalities on orthogonal geometries} \label{generalities}
\emph{Throughout this paper, $F$ will denote a field with characteristic different from $2$.  In fact, all the fields considered in this paper have characteristic zero.}  The results of this section are well-known and we will not repeat the proofs.  We refer the reader to \cite{Ar1988} and \cite{Di1971}.

Let $V$ be a finite dimensional $F$-vector space of dimension $n$.  If $V$ is equipped with a symmetric $F$-bilinear form $S$, then we say that $(V,S)$ is an orthogonal geometry.  If $S$ is clear from the context, we will refer to an orthogonal geometry just by $V$.  Instead of working with the symmetric $F$-bilinear form $S$, one can work with the associated quadratic form given by $q(v) = S(v,v)$ for all $v \in V$.  A pair $(V,q)$, where $V$ is a finite dimensional vector space over $F$ and $q$ is a quadratic form on $V$ is called a quadratic space.  We have a one-to-one correspondence between symmetric $F$-bilinear forms $S$ and quadratic forms $q$.  Given a quadratic form $q$, one recovers $S$ via the formula
$$S(v_{1},v_{2}) = \frac{1}{2}\left(q(v_{1} + v_{2}) - q(v_{1}) - q(v_{2}) \right).$$

Given an orthogonal geometry $V$, the radical of $V$, denoted by ${\rm Rad}(V)$, is defined as usual, i.e. it is the kernel of the linear transformation $V \longrightarrow V^{*}$ defined by $v \mapsto S(v, \cdot)$.  
\begin{Def}
Let $V$ be an orthogonal geometry.  Then $V$ is called non-singular if ${\rm Rad}(V) = 0$ and isotropic if ${\rm Rad}(V) = V$.  A vector $v \in V$ is called isotropic if $q(v) = 0$, otherwise non-isotropic.
\end{Def}
An orthogonal geometry $V$ is isotropic if and only if every vector $v \in V$ is isotropic.

Let $(V_{1},S_{1})$ and $(V_{2},S_{2})$ be two orthogonal geometries.  A linear transformation $f:V_{1} \longrightarrow V_{2}$ is called an orthogonal map if $S_{2}(f(v_{1}),f(v_{1}')) = S_{1}(v_{1},v_{1}')$ for all $v_{1},v_{1}' \in V_{1}$.  An orthogonal map $f:V_{1}\longrightarrow V_{2}$ is called an isometry if there exists an orthogonal map $g:V_{2}\longrightarrow V_{1}$ satisfying $f \circ g = id_{V_{2}}$ and $g \circ f = id_{V_{1}}$.  Note that if $f:V_{1} \longrightarrow V_{2}$ is a bijective orthogonal map, then it is an isometry.  More generally, an $F$-linear transformation $f:V_{1} \longrightarrow V_{2}$ is called an orthogonal similitude if there exists $\lambda \in F^{\times}$ such that $S_{2}(f(v_{1}),f(v_{1}')) = \lambda S_{1}(v_{1},v_{1}')$ for all $v_{1},v_{1}' \in V_{1}$.  The constant $\lambda \in F^{\times}$ is called the factor of similitude of $f$.  It is simple to check that if $f:V_{1} \longrightarrow V_{2}$ is an orthogonal similitude between two orthogonal geometries and $V_{1}$ is non-singular, then $f$ is necessarily injective.  

The set of isometries of an orthogonal geometry $V$ into itself forms a subgroup of the general linear group $GL(V)$ which is denoted by $O(V,S)$ or $O(V)$ if $S$ is understood from the context.  Moreover, we let $GO(V)$ be the group of orthogonal similitudes, that is
$$GO(V) = \{g \in GL(V) \, | \, S(g(v_{1}),g(v_{2})) = \lambda(g) S(v_{1},v_{2}) \text{ for some } \lambda(g) \in F^{\times} \}. $$
Note that the map $\lambda:GO(V) \longrightarrow F^{\times}$ is a group morphism and $O(V) = {\rm ker}(\lambda)$.  If $V$ is an orthogonal geometry over $F$ with symmetric $F$-bilinear form $S$ and $\lambda \in F^{\times}$, then we let $V^{\lambda}$ be the orthogonal geometry obtained from $V$ by rescaling the symmetric $F$-bilinear form $S$ by a factor $\lambda$.  
 
If $V$ is non-singular, then it is simple to check that any $\sigma \in O(V)$ satisfies ${\rm det}(\sigma) = \pm 1$.  The ones satisfying ${\rm det}(\sigma) = 1$ are called rotations and they form a subgroup of $O(V)$ which is denoted by $SO(V,S)$ or more simply by $SO(V)$.  The following result is well known.
\begin{Pro} \label{existenceort}
If $V$ is an orthogonal geometry, then $V = L_{1} \perp \ldots \perp L_{r}$, where $L_{i} = {\rm Span}(v_{i})$ are lines.  Moreover, $V$ is non-singular if and only if $v_{i}$ is a non-isotropic vector for all $i=1,\ldots,r$.
\end{Pro}
The set $\{v_{1},\ldots,v_{r} \}$ of the last proposition is called an orthogonal basis.  We will now recall the definition of some important isometries in $O(V)$.  An isometry $\sigma \in O(V)$ is called an involution if $\sigma^{2} = 1$.  If $\sigma$ is an involution, then we let
$$U = \left(\frac{1 - \sigma}{2}\right)  V \qquad\hbox{and}\qquad W = \left(\frac{1 + \sigma}{2}\right)  V. $$
(Recall that we are staying away from characteristic $2$.)  It is then a simple matter to show that $V = U \perp W$ and $\sigma = -id_{U} \perp id_{W}$.  The dimension of $U$ is called the type of $\sigma$.  An involution of type $1$ is called a symmetry with respect to the hyperplane $W$ or less precisely an hyperplane reflection or even more simply a reflection. 

If $\sigma = -id_{L} \perp id_{H}$ is an hyperplane reflection and $v \in L$ is a non-zero vector, then $v$ is a non-isotropic vector, since $L$ is non-singular.  On the other hand, if we start with a non-isotropic vector $v \in V$, then it is simple to check that $r_{v} \in O(V)$ given by
$$r_{v}(w) = w - 2 \cdot \frac{ S(w,v)}{S(v,v)}v, $$
whenever $w \in V$, is a symmetry with respect to the hyperplane $L^{\perp}$ where $L = {\rm Span}(v)$.  Conversely, every hyperplane reflection $-id_{L} \perp id_{H}$ is of the form $r_{v}$ for some non-isotropic vector $v \in L$.  Theorem \ref{CartanDieudonne} below is fundamental, but we first need the following lemma whose proof is left to the reader.

\begin{Lem} \label{goto}
Let $V$ be an orthogonal geometry and let $v,w \in V$.  If $q(v) = q(w) \neq 0$, then either
\begin{enumerate}
\item $q(v-w) \neq 0$, \label{unone}
\item $q(v+w) \neq 0$. \label{deuxtwo}
\end{enumerate}
In case (\ref{unone}), we have $r_{v-w}(v) = w$, and in case (\ref{deuxtwo}), we have $r_{w} \circ r_{v + w}(v) = w$.
\end{Lem}
We can now show:
\begin{The} \label{CartanDieudonne}
Let $V$ be a non-singular orthogonal geometry.  Then, every $\sigma \in O(V)$ is a product of hyperplane reflections.
\end{The}
\begin{proof}
Let $\sigma \in O(V)$.  By Proposition \ref{existenceort}, we know that $V = Fe_{1} \perp \ldots \perp Fe_{n}$ for some non-isotropic vectors $e_{i} \in V$.  Define $\psi_{i} \in O(V)$ inductively as follows:
\begin{equation*}
\psi_{i}=
\begin{cases}
r_{\psi_{i-1}\cdot \ldots \cdot \psi_{1} \cdot \sigma(e_{i}) - e_{i}}, & \text{if } q(\psi_{i-1}\cdot\ldots\psi_{1}\cdot \sigma(e_{i}) - e_{i}) \neq 0; \\
r_{e_{i}} \circ r_{\psi_{i-1}\cdot \ldots \cdot \psi_{1} \cdot \sigma(e_{i}) + e_{i}}, & \text{otherwise}.
\end{cases}
\end{equation*}
One checks using Lemma \ref{goto} that $\psi_{i}\psi_{i-1} \ldots \psi_{1} \sigma(e_{j}) = e_{j}$ for all $j=1,\ldots,i$.  It follows that $\sigma$ is a product of hyperplane reflections, and this is what we wanted to show.
\end{proof}

Assume now that $F$ is an ordered field, so that we can talk about the signature of an orthogonal geometry.  If $(V,S_{V})$ and $(W,S_{W})$ are two non-singular orthogonal geometries having the same signature, then they are not necessarily isometric in general.  But if every positive element of $F$ is a square in $F$ (for instance, if $F = \mathbb{R}$), then the possible signatures $(r,s)$ are in bijection with the isometry classes of non-singular orthogonal geometries.  In other words, if two non-singular orthogonal geometries $(V,S_{V})$ and $(W,S_{W})$ have the same signature, then they are isometric.

%\section{Clifford algebras, Pin and Spin groups}
\section{Clifford algebras, Pin and Spin groups} \label{clifpinspin}
\emph{By an $F$-algebra $\mathcal{A}$, we always mean a unital associative $F$-algebra and its unit element will be denoted by $1_{\mathcal{A}}$.  By a morphism of $F$-algebras, we always mean a morphism of unital associative $F$-algebras.  Throughout this section, $V$ will stand for a non-singular orthogonal geometry.  The symmetric $F$-bilinear form is denoted by $S$ and the associated quadratic form by $q$.}  
\begin{Def}
A universal Clifford algebra for the non-singular orthogonal geometry $V$ is an $F$-algebra $\mathcal{C}$ with an injective $F$-linear map $i:V \hookrightarrow \mathcal{C}$ satisfying
\begin{equation} \label{defClifford}
i(v)^{2} = - q(v) \cdot 1_{\mathcal{C}}
\end{equation}
and such that the following universal property holds true:  Given any $F$-algebra $\mathcal{A}$ and a $F$-linear map $f:V \longrightarrow \mathcal{A}$ satisfying $f(v)^{2} = - q(v) \cdot 1_{\mathcal{A}}$, there exists a unique $F$-algebra morphism $\tilde{f}:\mathcal{C} \longrightarrow \mathcal{A}$ such that $\tilde{f} \circ i = f$.
\end{Def}
We warn the reader that there is no consensus about the negative sign in (\ref{defClifford}).  Moreover, via the monomorphism $i:V \hookrightarrow \mathcal{C}$, we will identify $V$ as a linear subspace of $\mathcal{C}$.  We also identify $F$ with $F \cdot 1_{\mathcal{C}}$.  With these two identifications, the identity (\ref{defClifford}) becomes $v^{2} = -q(v)$.  We also have
\begin{equation} \label{commutatorrelation}
v \cdot w + w \cdot v = - 2S(v,w)
\end{equation}
for all $v,w \in V$ which reduces to (\ref{defClifford}) when $v = w$.  The existence of a universal Clifford algebra is standard and can be realized as a quotient of the tensor algebra $T(V)$.  Let $\{e_{1},\ldots,e_{n} \}$ be an orthogonal basis for $V$.  Then, (\ref{defClifford}) implies 
\begin{equation} \label{square}
e_{i}^{2} = - q(e_{i}),
\end{equation}
for all $i=1,\ldots,n$.  Since $V$ is assumed to be non-singular, we have $q(e_{i}) \neq 0$, and therefore $e_{i} \in \mathcal{C}^{\times}$ for all $i=1,\ldots,n$.  Moreover, (\ref{commutatorrelation}) implies 
\begin{equation} \label{orto}
e_{i} \cdot e_{j} = -e_{j} \cdot e_{i},
\end{equation}
whenever $i \neq j$.  We let $\Omega = \{(i_{1},\ldots,i_{s}) \, | \, 1 \le s \le n  \text{ and } 1 \le i_{1} < \ldots < i_{s} \le n\} \cup \{\varnothing \}$.  Clearly, $|\Omega| = 2^{n}$.  Given $I = (i_{1},\ldots,i_{s}) \in \Omega$, we set $e_{I} = e_{i_{1}} \cdot \ldots \cdot e_{i_{s}} \in \mathcal{C}^{\times}$ and we agree that $e_{\varnothing} = 1_{\mathcal{C}}$.  It is well-known that $\{e_{I} \, | \, I \in \Omega\}$ is a basis for $\mathcal{C}$ considered as an $F$-vector space.  Therefore, ${\rm dim}_{F}(\mathcal{C}) = 2^{n}$.

As a consequence of the universal property satisfied by a universal Clifford algebra, we have the following result.
\begin{The} \label{orthomap}
Let $V_{1}$ and $V_{2}$ be two non-singular orthogonal geometries and let $f:V_{1} \longrightarrow V_{2}$ be an orthogonal map.  If $\mathcal{C}_{i}$ is a universal Clifford algebra for $V_{i}$ for $i=1,2$, then there is a unique $F$-algebra morphism $\tilde{f}:\mathcal{C}_{1} \longrightarrow \mathcal{C}_{2}$ making the following diagram commutative:
\begin{equation*}
\xymatrix{
V_{1}\, \ar[d]^f \ar@{^{(}->}[r] & \mathcal{C}_{1} \ar[d]^{\tilde{f}} \\
V_{2}\, \ar@{^{(}->}[r] & \mathcal{C}_{2}
}
\end{equation*}
\end{The}
We point out in passing that both $f$ and $\tilde{f}$ are necessarily injective, since $V_{1}$ is assumed to be non-singular.

From Theorem \ref{orthomap} and a slight variation of it follow the existence of the principal involution and anti-involution.  The principal involution on $\mathcal{C}$ will be denoted by $x \mapsto x'$.  We remind the reader that it is the unique $F$-algebra automorphism on $\mathcal{C}$ satisfying $v' = -v$, whenever $v \in V$.  The principal anti-involution on $\mathcal{C}$ will be denoted by $x \mapsto x^{*}$. It is the unique $F$-algebra anti-automorphism of $\mathcal{C}$ satisfying $v^{*} = v$, whenever $v \in V$.  It is simple to check that the principal involution and anti-involution commute, that is $(x^{*})' = (x')^{*}$.  At last, we will make use of the Clifford conjugation defined for $x \in \mathcal{C}$ by $\overline{x} = (x^{*})'$.  The Clifford conjugation is also an anti-involution and if $v \in V$, then $\overline{v} = -v$.

Given a universal Clifford algebra $\mathcal{C}$ for $V$, we let
$$\mathcal{C}^{0} = \{x \in \mathcal{C} \, | \, x' = x \} \qquad\hbox{and}\qquad \mathcal{C}^{1} = \{x \in \mathcal{C} \, | \, x' = -x\}. $$
Note that $V \subseteq \mathcal{C}^{1}$ and that $\mathcal{C}^{0}$ is a $F$-subalgebra of $\mathcal{C}$.  Hence, $\mathcal{C}$ has a $\mathbb{Z}/2\mathbb{Z}$-grading.  If $\{e_{1},\ldots,e_{n} \}$ is an orthogonal basis for $V$, then $e_{I} \in \mathcal{C}^{0}$ if and only if $|I|$ is even and $e_{I} \in \mathcal{C}^{1}$ if and only if $|I|$ is odd.

It is well-known that $Z_{\mathcal{C}}(C^{0}) = F + F \cdot e_{\Sigma}$, where $\Sigma = (1,\ldots,n)$.  Moreover, the center of $\mathcal{C}$ is
\begin{equation} \label{center}
Z(\mathcal{C}) = 
\begin{cases}
F + Fe_{\Sigma}, &\text{ if } n \text{ is odd}, \\
F, &\text{ if } n \text{ is even}.
\end{cases}
\end{equation}

%\subsection{The Clifford group}
\subsection{The Clifford group} \label{sectionClifford}
As before, let $V$ be a non-singular orthogonal geometry and let $\mathcal{C} = \mathcal{C}(V)$ be a universal Clifford algebra for $V$.  We have an obvious action of $\mathcal{C}^{\times}$ on $\mathcal{C}$ given by conjugation, that is $x * y = x y x^{-1}$, whenever $x \in \mathcal{C}^{\times}$ and $y \in \mathcal{C}$.  The associated Clifford group is $R(V) = \{x \in \mathcal{C}^{\times} \, | \, x * V \subseteq V \}$.  Given $x \in R(V)$, we get a $\chi(x) \in GL(V)$ defined by $\chi(x)(v) = xvx^{-1}$, whenever $v \in V$.  In fact, it is simple to check that $\chi(x) \in O(V)$ so that we get an $F$-linear representation $\chi:R(V) \longrightarrow O(V)$.  This linear representation is surjective if ${\rm dim}(V)$ is even, but if ${\rm dim}(V)$ is odd, then $\chi(R(V)) \subseteq SO(V)$.  Moreover, the kernel of $\chi$ is different depending on the parity of ${\rm dim}(V)$.  This phenomena is due to the fact that if $v \in V$ is a non-isotropic vector, then 
\begin{equation} \label{negconjugation}
\chi(v) = -r_{v}.
\end{equation}
Since the landmark paper \cite{AtBoSh64}, it has been realized that it is nicer to work with a different action than the one given by conjugation.  This will give us a different Clifford group, denoted by $\Gamma(V)$, and a representation $\rho:\Gamma(V) \longrightarrow O(V)$ which will always be surjective and the kernel will be the same independently of the parity of ${\rm dim}(V)$.  The key point is that instead of (\ref{negconjugation}), we will have $\rho(v) = r_{v}$, whenever $v$ is a non-isotropic vector of $V$.  

The action of Atiyah, Bott and Shapiro of $\mathcal{C}^{\times}$ on $\mathcal{C}$ is given by $x * y = x y (x')^{-1}$, whenever $x \in \mathcal{C}^{\times}$ and $y \in \mathcal{C}$.  The Clifford group $\Gamma(V)$ is given by $\Gamma(V) = \{x \in \mathcal{C}^{\times} \, | \, x * V \subseteq V \}$.  Hence $\Gamma(V)$ acts on $V$ and it is simple to check that the action of $\Gamma(V)$ on $V$ is $F$-linear.  We then get an $F$-linear representation $\Gamma(V) \longrightarrow GL(V)$.  Note that if $x \in \Gamma(V)$ and $v \in V$, then by definition $xv(x')^{-1} \in V$.  Applying the principal involution gives $x'vx^{-1} = xv(x')^{-1}$ whenever $x \in \Gamma(V)$ and $v \in V$.  It is then a simple calculation to show that $\rho(\Gamma(V)) \subseteq O(V)$; hence, we have an $F$-linear representation
\begin{equation} \label{Cliffordrep}
\rho:\Gamma(V) \longrightarrow O(V).
\end{equation}
Moreover, it is simple to check that any non-isotropic vector is in $\Gamma(V)$, and given such a vector $v$, one has $\rho(v) = r_{v}$.  It follows from Theorem \ref{CartanDieudonne} that the representation (\ref{Cliffordrep}) is surjective.
\begin{Pro} \label{kernel}
With the notation as above, we have ${\rm ker}(\rho) = F^{\times}$.  Therefore, we have a short exact sequence
$$1 \longrightarrow F^{\times} \longrightarrow \Gamma(V) \stackrel{\rho}{\longrightarrow} O(V) \longrightarrow 1.$$
\end{Pro}
\begin{proof}
Let $x \in {\rm ker}(\rho)$, then we have $xv = v x'$ for all $v \in V$.  Writing $x = x_{0} + x_{1}$ for some $x_{i} \in \mathcal{C}^{i}$, we then have $x_{0}v = vx_{0}$ and $x_{1}v = -vx_{1}$ for all $v \in V$.  This means that $x_{0} \in Z(\mathcal{C}) \cap \mathcal{C}^{0}$.  Using (\ref{center}), one concludes that $x_{0} \in F^{\times}$. On the other hand, a simple computation shows that if $x \in \mathcal{C}$ and $xv = -vx$ for all $v \in V$, then $x = 0$ if $n$ is odd and $x = \lambda e_{\Sigma}$ for some $\lambda \in F$ if $n$ is even.  Since in our situation $x_{1} \in \mathcal{C}^{1}$, we necessarily have $x_{1} = 0$.  This completes the proof.
\end{proof}
It is also worthwhile to point out the following corollary.
\begin{Cor} \label{clifford_generate}
The Clifford group $\Gamma(V)$ is generated by the non-isotropic vectors in $V$.
\end{Cor}

As is customary, we define $\Gamma^{0}(V) = \Gamma(V) \cap \mathcal{C}^{0}$ and $\Gamma^{1}(V) = \Gamma(V) \cap \mathcal{C}^{1}$.
\begin{Pro}
The Clifford group $\Gamma(V)$ is the disjoint union of $\Gamma^{0}(V)$ and $\Gamma^{1}(V)$.
\end{Pro}
\begin{proof}
This follows from Corollary \ref{clifford_generate}.
\end{proof}
\begin{Pro}
With the notation as above, the restriction of $\rho$ to $\Gamma^{0}(V)$ (which is a subgroup of $\Gamma(V)$), denoted by $\rho_{0}$, induces a short exact sequence
$$1 \longrightarrow F^{\times} \longrightarrow \Gamma^{0}(V) \stackrel{\rho_{0}}{\longrightarrow} SO(V) \longrightarrow 1. $$
\end{Pro}
\begin{proof}
Since $F^{\times} \subseteq \Gamma^{0}(V)$, the only thing we have to show is that if $x \in \Gamma(V)$ is such that $\rho(x) = \sigma \in SO(V)$, then $x \in \Gamma^{0}(V)$.  But, if so, then Theorem \ref{CartanDieudonne} implies that there exist non-isotropic vectors $v_{1},\ldots,v_{m} \in V$ such that $\sigma = r_{v_{1}} \cdot \ldots \cdot r_{v_{m}}$.  Since, $\sigma \in SO(V)$, we have ${\rm det}(\sigma) = 1$, and thus $(-1)^{m} =1$ meaning that $m$ is even.  Now, there exists $\lambda \in F^{\times}$ such that $x = \lambda v_{1} \cdot \ldots \cdot v_{m}$.  The integer $m$ being even, we conclude that $x \in \Gamma^{0}(V)$ as we wanted to show.
\end{proof}

%\subsection{Abstract Pin and Spin groups}
\subsection{Abstract Pin and Spin groups} \label{abstract}
Pin and Spin groups will appear in various disguises throughout this paper, and it is convenient to give their axiomatic definitions.
\subsubsection{Pin groups} \label{abstractpin}
Suppose we are given:
\begin{enumerate}
\item A group $G$ and a non-singular orthogonal geometry $V$ over $F$.
\item A short exact sequence of groups $1 \longrightarrow F^{\times} \longrightarrow G \stackrel{\rho}{\longrightarrow} O(V) \longrightarrow 1$.
\item A group morphism $N:G \longrightarrow F^{\times}$ satisfying $N(\lambda) = \lambda^{2}$ whenever $\lambda \in F^{\times}$.
\end{enumerate}
Then, we get a commutative diagram
\begin{equation*}
        \begin{CD}
                1 @>>> F^{\times} @>>> G @>{\rho}>> O(V) @>>>1 \\
                & & @VV{f}V @VV{N}V @VV{\vartheta}V \\
                1  @>>> F^{\times 2}@>>> F^{\times} @>>> F^{\times}/F^{\times 2} @>>> 1,\\
        \end{CD}
\end{equation*}
where $f:F^{\times} \longrightarrow F^{\times 2}$ is given by $x \mapsto f(x) = x^{2}$ and $\vartheta$ is induced from $N$ so that if $\rho(g) = \sigma$, then $\vartheta(\sigma) = N(g) \cdot F^{\times 2}$.
\begin{Def}
One defines the group $Pin^{+}(\rho,N) = {\rm ker}(N)$.
\end{Def}
Since the map $f$ is surjective, the snake lemma gives the following diagram
\begin{equation*}
        \begin{CD}
                1 @>>> \{\pm 1 \} @>>> Pin^{+}(\rho,N) @>>> {\rm ker}(\vartheta) @>>> 1\\
                & & @VVV @VVV @VVV \\
                1 @>>> F^{\times} @>>> G @>{\rho}>> O(V) @>>>1 \\
                & & @VV{f}V @VV{N}V @VV{\vartheta}V \\
                1  @>>> F^{\times 2}@>>> F^{\times} @>>> F^{\times}/F^{\times 2} @>>> 1,\\
        \end{CD}
\end{equation*}
whose rows are exact.  This last diagram induces in turns the following important exact sequence:
\begin{equation} \label{pinplusexact}
1 \longrightarrow \{\pm 1 \} \longrightarrow Pin^{+}(\rho,N) \stackrel{\rho}{\longrightarrow} O(V) \stackrel{\vartheta}{\longrightarrow} F^{\times}/F^{\times 2}.
\end{equation}
The group morphism $\vartheta:O(V) \longrightarrow F^{\times}/F^{\times 2}$ is called the spinor norm morphism.

%\subsubsection{Spin groups}
\subsubsection{Spin groups} \label{abstractspin}
A similar theory can be developed for the group $SO(V)$ instead of $O(V)$.  Suppose we are given:
\begin{enumerate}
\item A group $G$ and a non-singular orthogonal geometry $V$ over $F$.
\item A short exact sequence of groups $1 \longrightarrow F^{\times} \longrightarrow G \stackrel{\rho}{\longrightarrow} SO(V) \longrightarrow 1$.
\item A group morphism $N:G \longrightarrow F^{\times}$ satisfying $N(\lambda) = \lambda^{2}$ whenever $\lambda \in F^{\times}$.
\end{enumerate}
If one defines the group $Spin^{+}(\rho,N) = {\rm ker}(N)$, then we have the important exact sequence
\begin{equation} \label{spinplusexact}
1 \longrightarrow \{\pm 1 \} \longrightarrow Spin^{+}(\rho,N) \stackrel{\rho}{\longrightarrow} SO(V) \stackrel{\vartheta}{\longrightarrow} F^{\times}/F^{\times 2},
\end{equation}
and the group morphism $\vartheta:SO(V) \longrightarrow F^{\times}/F^{\times 2}$ is also called the spinor norm.

%\subsection{Pin and Spin groups}
\subsection{Pin and Spin groups} \label{pinspin}
Let $V$ be a non-singular orthogonal geometry and let $\mathcal{C}$ be a universal Clifford algebra for $V$.  As in the previous sections, we let $\Gamma(V)$ be the Clifford group.  If $x \in \Gamma(V)$, then $x \overline{x} \in F^{\times}$.  Indeed, it follows from Corollary \ref{clifford_generate} that any $x \in \Gamma(V)$ can be written as $x = v_{1} \cdot \ldots \cdot v_{m}$ for some non-isotropic vectors $v_{i} \in V$.  Hence, $x \overline{x} = q(v_{1}) \cdot \ldots \cdot q(v_{m}) \in F^{\times}$.  We obtain a group morphism $N:\Gamma(V) \longrightarrow F^{\times}$ defined by $x \mapsto N(x) = x \overline{x}$.  Note that if $x \in \Gamma(V)$, then $x  \overline{x} = \overline{x}  x$.  Moreover, if $\lambda \in F^{\times}$, then $N(\lambda) = \lambda^{2}$.  We are now in the setting of \S \ref{abstract}, and we let $Pin^{+}(V) = Pin^{+}(\rho,N)$ and $Spin^{+}(V) = Spin^{+}(\rho_{0},N)$.

\begin{Def}
From now on, we let 
$$O^{+}(V) = \rho(Pin^{+}(V))\qquad\hbox{and}\qquad SO^{+}(V) = \rho(Spin^{+}(V)).$$
\end{Def}
The groups $O^{+}(V)$ and $SO^{+}(V)$ are sometimes called the spinorial kernels.

%\subsection{Lorentzian geometries over $\mathbb{R}$}
\subsection{Lorentzian geometry over $\mathbb{R}$}
Let $V$ be the real vector space $\mathbb{R}^{m}$ and let $p,q \in \mathbb{Z}_{\ge 0}$ be such that $p+q = m$.  The function $S:V \times V \longrightarrow \mathbb{R}$ defined by
\begin{equation} \label{pq}
S(x,y) = \sum_{i=1}^{p}x_{i}y_{i} - \sum_{j=p+1}^{m}x_{j}y_{j}, 
\end{equation}
where $x = (x_{1},\ldots,x_{m})$ and $y = (y_{1},\ldots,y_{m})$, is easily seen to be a symmetric $\mathbb{R}$-bilinear form of signature $(p,q)$.  Moreover, the orthogonal geometry $(V,S)$ is non-singular.  As we pointed out is \S \ref{generalities}, these are the only non-singular orthogonal geometries over $\mathbb{R}$ (up to isometry).  These orthogonal geometries will be denoted by $\mathbb{R}^{p,q}$.

One of them will be particularly important for us.  It is $\mathbb{R}^{n,1}$ which will be referred to as the Lorentzian geometry.  For the reminder of this section, we let $S$ denote the symmetric bilinear form (\ref{pq}) when $p = n$ and $q = 1$.  The corresponding quadratic form will be denoted by $q$.

In this situation, we have $O(\mathbb{R}^{n,1})$ and if one chooses the standard basis $(e_{1},\ldots,e_{n})$ for $\mathbb{R}^{n+1}$ which is an orthogonal basis for $\mathbb{R}^{n,1}$, then we we have the corresponding matrix groups.  If 
$$J = diag(1,\ldots,1,-1) \in M_{n+1}(\mathbb{R}), $$
then we have
$$O(n,1) = \{M \in GL(n+1,\mathbb{R}) \, | \, M^{t}JM = J \}. $$
The set $C = \{x \in \mathbb{R}^{n+1} \, | \, q(x) \le 0 \}$ is a double-cone, and with
$$C_{\pm} = \{x \in C \, | \, \pm x_{n+1} \ge ||(x_{1},x_{2},\ldots,x_{n}) || \}, $$
one defines
$$O^{+}(\mathbb{R}^{n,1}) = \{g \in O(\mathbb{R}^{n,1}) \, | \, g(C_{+}) = C_{+} \}. $$
The reader will notice that we already defined a group $O^{+}(\mathbb{R}^{n,1})$ in \S \ref{pinspin}, but there are no ambiguities, since one can check that both groups are the same.  Moreover, it is simple to check that if $M = (m_{ij}) \in O(n,1)$, then $m_{n+1,n+1} \ge 1$ or $m_{n+1,n+1} \le -1$, and one has
$$O^{+}(n,1) = \{M \in O(n,1) \, | \, m_{n+1,n+1} \ge 1 \}. $$
One also lets $SO^{+}(\mathbb{R}^{n,1})= O^{+}(\mathbb{R}^{n,1}) \cap SL(\mathbb{R}^{n+1})$ which one can check is equal to the group $SO^{+}(\mathbb{R}^{n,1})$ defined in \S \ref{pinspin}.  As explained in \S \ref{pinspin}, we have the two $2$-covers
\begin{equation} \label{Oplus}
1 \longrightarrow \{\pm 1 \} \longrightarrow Pin^{+}(\mathbb{R}^{n,1}) \stackrel{\rho}{\longrightarrow} O^{+}(\mathbb{R}^{n,1}) \longrightarrow 1, 
\end{equation}
and
\begin{equation} \label{SOplus}
1 \longrightarrow \{\pm 1 \} \longrightarrow Spin^{+}(\mathbb{R}^{n,1}) \stackrel{\rho}{\longrightarrow} SO^{+}(\mathbb{R}^{n,1}) \longrightarrow 1.
\end{equation}

%\subsection{Change of fields}
\subsection{Change of fields} \label{changeoffields}
Throughout this section, we let $F$ be our fixed ground field, and we let $E$ be a field extension of $F$.  If $V$ is an $F$-vector space, then we let $V_{E} = E \otimes_{F} V$ be the $E$-vector space obtained from $V$ by extending the scalars to $E$.  Given a $K$-vector space $V$, where $K$ is a field, we let $\mathcal{L}(V)$ denote the $K$-algebra of $K$-linear endomorphisms of $V$.  The universal property satisfied by $V_{E}$ induces an injective morphism of $F$-algebras
\begin{equation} \label{algebraemb}
\mathcal{L}(V) \hookrightarrow \mathcal{L}(V_{E}). 
\end{equation}
The image of an $f \in \mathcal{L}(V)$ via this map will be denoted by $f_{E}$ which satisfies on pure tensors the formula $f_{E}(e \otimes v) = e \otimes f(v)$.  The morphism (\ref{algebraemb}) induces in turn an injective group morphism
\begin{equation} \label{embedofgenerallineargroup}
GL(V) \hookrightarrow GL(V_{E}). 
\end{equation}
Now, if $V$ is an orthogonal geometry over $F$, say with associated symmetric $F$-bilinear form $S$ and quadratic form $q$, then $V_{E}$ becomes an orthogonal geometry over $E$ whose symmetric $E$-bilinear form $S_{E}$ is given on pure tensors by $S_{E}(e_{1} \otimes v_{1},e_{2} \otimes v_{2}) = e_{1} \cdot e_{2} \cdot S(v_{1},v_{2})$.  The associated quadratic form $q_{E}$ satisfies $q_{E}(e \otimes v) = e^{2} q(v)$.  If $V$ is non-singular, then so is $V_{E}$, and we now assume that.  Given $f \in O(V)$, it is simple to check that $f_{E} \in O(V_{E})$, and therefore the group morphism (\ref{embedofgenerallineargroup}) induces an injective group morphism
\begin{equation} \label{embedorthogonal}
\tau_{E}:O(V) \hookrightarrow O(V_{E}).
\end{equation}

We now let $\mathcal{C}(V)$ be a universal Clifford algebra for $V$, and $\mathcal{C}(V_{E})$ will be one for $V_{E}$.  The injective morphism of $F$-vector spaces $V \hookrightarrow V_{E} \hookrightarrow \mathcal{C}(V_{E})$, satisfies $(1 \otimes v)^{2} =  - q_{E}(1 \otimes v) = - q(v)^{2}$, and therefore, we get from the universal property satisfied by the Clifford algebra $\mathcal{C}(V)$ a morphism of $F$-algebras $\psi_{E}:\mathcal{C}(V) \hookrightarrow \mathcal{C}(V_{E})$ which is in fact injective.  Hence, we can view $\mathcal{C}(V)$ inside $\mathcal{C}(V_{E})$ using $\psi_{E}$.  Moreover, the map $\psi_{E}$ behaves well with respect to the three involutions, namely
$$\psi_{E}(x') = \psi_{E}(x)', \psi_{E}(x^{*}) = \psi_{E}(x)^{*}, \text{ and } \psi_{E}(\overline{x}) = \overline{\psi_{E}(x)},  $$
for all $x \in \mathcal{C}(V)$.  Since $\psi_{E}$ behaves well in particular with respect to the principal involution, we have $\psi_{E}(\Gamma(V)) \subseteq \Gamma(V_{E})$.  Therefore, we get the following commutative diagram:
\begin{equation*}
        \begin{CD}
                1 @>>> F^{\times} @>>> \Gamma(V) @>{\rho}>> O(V) @>>>1 \\
                & & @VVV @VV{\psi_{E}}V @VV{\tau_{E}}V \\
                1  @>>> E^{\times }@>>> \Gamma(V_{E}) @>{\rho_{E}}>> O(V_{E}) @>>> 1\\
        \end{CD}
\end{equation*}
%and
%\begin{equation*}
%        \begin{CD}
%                1 @>>> F^{\times} @>>> \Gamma^{0}(V) @>{\rho}>> SO(V) @>>>1 \\
%                & & @VVV @VV{\psi_{E}}V @VV{\tau_{E}}V \\
%                1  @>>> E^{\times }@>>> \Gamma^{0}(V_{E}) @>{\rho_{E}}>> SO(V_{E}) @>>> 1\\
%        \end{CD}
%\end{equation*}
where the vertical arrows are all injective. 

Now, since $\psi_{E}$ behaves well with respect to the Clifford conjugation, we obviously have
$$ \psi_{E}(Pin^{+}(V)) \subseteq Pin^{+}(V_{E}).$$ 
Moreover, we get the following commutative diagram:
\begin{equation*}
        \begin{CD}
                1 @>>> \{\pm 1\} @>>> Pin^{+}(V) @>{\rho}>> O(V) @>{\vartheta}>> F^{\times}/F^{\times 2} \cdot \{ \pm 1 \} \\
                & & @VVV @VV{\psi_{E}}V @VV{\tau_{E}}V @VVV \\
                1  @>>> \{\pm 1 \} @>>> Pin^{+}(V_{E}) @>{\rho_{E}}>> O(V_{E}) @>{\vartheta_{E}}>> E^{\times}/E^{\times 2}\cdot \{\pm 1\}\\
        \end{CD}
\end{equation*}
where the first three vertical arrows are injective, but not necessarily the last one.  One has a similar diagram for $Spin^{+}(V)$.

%\section{Vahlen groups}
\section{Vahlen groups} \label{Vahlen1}
In \cite{Vahlen02}, Vahlen described the group of orientation preserving isometries of the $n$-dimensional hyperbolic space as the central quotient of a certain group of two by two matrices with entries in the Clifford algebra of $\mathbb{R}^{n-2,0}$.  These groups can be viewed as generalizations of both $SL(2,\mathbb{R})$ and $SL(2,\mathbb{C})$, since $PSL(2,\mathbb{R})$ and $PSL(2,\mathbb{C})$ are the groups of orientation preserving isometries of the $2$-dimensional and $3$-dimensional hyperbolic spaces respectively.  Vahlen's work had been forgotten for a while until Maass used them in his fundamental paper \cite{Maass49}.  These groups, now called Vahlen groups, were studied later by Ahlfors in \cite{Ahlfors84}, \cite{Ahlfors85}, \cite{Ahlfors85a}, \cite{Ahlfors86} and \cite{Ahlfors86a} in connection with the group of M\"{o}bius transformations $M(\mathbb{R}^{n})$.  In \cite{ElGrMe87}, it was shown how to define a Vahlen group for any non-singular orthogonal geometry over any field of characteristic different from $2$, and not just over $\mathbb{R}$ as it had been done previously.  Moreover, they showed that a Vahlen group is isomorphic to a certain spin group, and this systematically gave isomorphisms between classical groups in small dimensions, a subject which had been previously studied by van der Waerden and Dieudonn\'{e} among others.  

In the literature, one can find the definition of Vahlen groups for the so-called paravectors and also for non-paravectors.  In this section, we place ourselves in the latter situation, and we define Vahlen groups in the non-paravector situation for any non-singular orthogonal geometry over any field of characteristic different from $2$.  Our approach is via pin and spin groups.  This might be known to the experts, but we have not found it in the literature, so we include these results here.  In the recent preprint \cite{McInroy}, McInroy describes Vahlen groups over commutative rings, not only fields, using an approach which is very similar to ours.  The paravector case and the relationship between the two setups will be explained in \S \ref{connections} below.

%\subsection{Vahlen's groups}
\subsection{Vahlen groups} \label{Vahlen}
We start with a base field $F$ of characteristic different from $2$ and we recall the following important definition.
\begin{Def}
A non-singular plane (meaning a two dimensional $F$-vector space) with an orthogonal geometry is called a hyperbolic plane if it contains a non-zero isotropic vector.
\end{Def}
The following lemma is simple and the proof is left to the reader.
\begin{Lem}
Let $P$ be a hyperbolic plane $P$ and assume that $f_{1} \in P$ is a non-zero isotropic vector.  Then, there exists a unique non-zero isotropic vector $f_{2}$ satisfying
$$S(f_{1},f_{2}) = -\frac{1}{2}.$$
\end{Lem}
We will call such a pair $(f_{1},f_{2})$ a hyperbolic pair.  Note that a hyperbolic pair $(f_{1},f_{2})$ is necessarily a basis for the hyperbolic plane $P$ and that the isotropic vectors in $P$ consist precisely of ${\rm Span}(f_{1}) \cup {\rm Span}(f_{2})$.  We remark as well that two hyperbolic planes are necessarily isometric.

Let $(V,S_{1})$ be a non-singular orthogonal geometry (with associated quadratic form $q_{1}$) and let $(P,S_{2})$ be a hyperbolic plane (with associated quadratic form $q_{2}$).  For the remainder of \S \ref{Vahlen}, we set
$$W = V \perp P, $$
and we fix a hyperbolic pair $(f_{1},f_{2})$ for $P$.  Note that $W$ is also a non-singular orthogonal geometry.  We let $S$ denote its symmetric $F$-bilinear form, and we let $q$ denote the corresponding quadratic form.  Every $w \in W$ can be uniquely written as $w = v + \lambda_{1}f_{1} + \lambda_{2}f_{2}$ for some $v \in V$ and $\lambda_{1},\lambda_{2} \in F$.  Note that $q_{2}(\lambda_{1}f_{1} + \lambda_{2}f_{2}) = - \lambda_{1}\lambda_{2}$, and therefore
$$q(v + \lambda_{1}f_{1} + \lambda_{2}f_{2}) = q_{1}(v) - \lambda_{1}\lambda_{2}. $$
Following \cite{Sh04}, we define a map $\phi:W \longrightarrow M_{2}(\mathcal{C}(V))$ via 
\begin{equation*}
w \mapsto\phi(w) =
\begin{pmatrix}
v & \lambda_{1} \\
\lambda_{2} & \bar{v}
\end{pmatrix},
\end{equation*}
where $w = v + \lambda_{1}f_{1} + \lambda_{2}f_{2}$.  The map $\phi$ is clearly $F$-linear, and a simple computation shows that
\begin{equation*}\phi(w)^{2} = -q(w) \cdot I_{2} \qquad\hbox{where}\qquad
I_{2} =
\begin{pmatrix}
1 & 0 \\
0 & 1
\end{pmatrix}
\end{equation*}
is the unit element of the $F$-algebra $M_{2}(\mathcal{C}(V))$.  By the universal property satisfied by universal Clifford algebras, we get an $F$-algebra morphism
$$\phi:\mathcal{C}(W) \longrightarrow M_{2}(\mathcal{C}(V)), $$
which we denote by the same symbol $\phi$.  It is simple to check that $\phi$ is surjective, and since $\mathcal{C}(W)$ and $M_{2}(\mathcal{C}(V))$ have the same dimensions as $F$-vector spaces, the morphism $\phi$ is an isomorphism.  The map $\phi$ being an isomorphism of $F$-algebras, we have in particular $\phi(\mathcal{C}(W)^{\times}) = M_{2}(\mathcal{C}(V))^{\times}$.  We can now define the notion of Vahlen groups.
\begin{Def}
We define the following subgroups of $M_{2}(\mathcal{C}(V))^{\times}$.
\begin{enumerate}
\item $\mathcal{V}(V) = \phi(\Gamma(W))$,
\item $\mathcal{V}^{0}(V) = \phi(\Gamma^{0}(W))$
\end{enumerate}
\end{Def}
We warn the reader that this notation is our own  and we have not seen it in the literature.  Our goal now is to give a more explicit description of Vahlen groups.  The three (anti) involutions $',^{*}, \bar{}$  of $\mathcal{C}(W)$ correspond to some (anti) involutions of $M_{2}(\mathcal{C}(V))$ which we will denote by $\alpha, \beta, \gamma$ respectively.  Our next task is to find formulas for $\alpha, \beta$ and $\gamma$.  

Given 
\begin{equation*}
A =
\begin{pmatrix}
a & b \\
c & d
\end{pmatrix} \in M_{2}(\mathcal{C}(V))
\end{equation*}
we set
\begin{equation*}
\alpha(A)
=
\begin{pmatrix}
a' & -b' \\
-c' & d'
\end{pmatrix},
\quad
\beta(A)
=
\begin{pmatrix}
\bar{d} & \bar{b} \\
\bar{c} & \bar{a}
\end{pmatrix}
\quad\hbox{and}\quad
\gamma(A)
=
\begin{pmatrix}
d^{*} & -b^{*} \\
-c^{*} & a^{*}
\end{pmatrix}.
\end{equation*}

\begin{Lem} \label{alpbetgam}
We have $\phi(x') = \alpha(\phi(x))$, $\phi(x^{*}) = \beta(\phi(x))$ and $\phi(\bar{x}) = \gamma(\phi(x))$ 
for all $x \in \mathcal{C}(W)$. 
\end{Lem}
\begin{proof}
To check the first equation we just have to show the following three properties:
\begin{enumerate}
\item $\alpha^{2} = id_{M_{2}(\mathcal{C}(V))}$,
\item $\alpha(A) = -A$ whenever $A \in \phi(W)$,
\item $\alpha(A \cdot B) = \alpha (A) \cdot \alpha(B)$ for all $A, B \in M_{2}(\mathcal{C}(V))$.
\end{enumerate}
To check the second equation we just have to show the following three properties:
\begin{enumerate}
\item $\beta^{2} = id_{M_{2}(\mathcal{C}(V))}$,
\item $\beta(A) = A$ whenever $A \in \phi(W)$,
\item $\beta(A \cdot B) = \beta(B) \cdot \beta(A)$ for all $A, B \in M_{2}(\mathcal{C}(V))$.
\end{enumerate}
We leave these simple computations to the reader. The third equation follows from the first two.
\end{proof}

The following two results are now easy to check.

\begin{Lem} \label{evenodd}
Let 
\begin{equation*}
A =
\begin{pmatrix}
a & b \\
c & d
\end{pmatrix} \in M_{2}(\mathcal{C}(V)).
\end{equation*}
Then, 
\begin{enumerate}
\item $A \in \phi(\mathcal{C}^{0}(W))$ if and only if $a,d \in \mathcal{C}^{0}(V)$ and $b,c \in \mathcal{C}^{1}(V)$,
\item $A \in \phi(\mathcal{C}^{1}(W))$ if and only if $a,d \in \mathcal{C}^{1}(V)$ and $b,c \in \mathcal{C}^{0}(V)$.
\end{enumerate}
\end{Lem}

\begin{Lem} \label{invertible}
Let 
\begin{equation*}
A =
\begin{pmatrix}
a & b \\
c & d
\end{pmatrix} \in M_{2}(\mathcal{C}(V)),
\end{equation*}
be such that
\begin{enumerate}
\item $ba^{*} - ab^{*} = cd^{*} - dc^{*} = 0$,
\item $a^{*}c - c^{*}a = d^{*}b - b^{*}d = 0$.
\end{enumerate}
If $ad^{*} - bc^{*} = d^{*}a - b^{*}c = \lambda \in F^{\times}$, then $A \in M_{2}(\mathcal{C}(V))^{\times}$ and 
\begin{equation*}
A^{-1} =  \frac{1}{\lambda}
\begin{pmatrix}
d^{*} & -b^{*} \\
-c^{*} & a^{*}
\end{pmatrix}.
\end{equation*}
\end{Lem}

We will now give a characterization of the $A \in M_{2}(\mathcal{C}(V))$ which lie in $\mathcal{V}(V)$.
\begin{The} \label{vahlenexpand}
Let 
\begin{equation*}
A =
\begin{pmatrix}
a & b \\
c & d
\end{pmatrix} \in M_{2}(\mathcal{C}(V)).
\end{equation*}
Then $A \in \mathcal{V}(V)$ if and only if the following conditions are satisfied:
\begin{enumerate}
\item $ad^{*} - bc^{*} = d^{*}a - b^{*}c = \lambda \in F^{\times}$, \label{un}
\item $ba^{*} - ab^{*} = cd^{*} - dc^{*} = 0$, \label{deux}
\item $a^{*}c - c^{*}a = d^{*}b - b^{*}d = 0$, \label{trois}
\item $a \bar{a}, b \bar{b}, c \bar{c}, d \bar{d} \in F$, \label{quatre}
\item $b \bar{d}, a \bar{c} \in V$, \label{cinq}
\item $av\bar{b} + b\bar{v}\bar{a}, cv\bar{d} + d\bar{v}\bar{c} \in F$ for all $v \in V$, \label{six}
\item $av\bar{d} + b\bar{v} \bar{c} \in V$ for all $v \in V$. \label{sept}
\end{enumerate}
Moreover $A \in \mathcal{V}^{0}(V)$ if and only if (\ref{un}) through (\ref{sept}) are satisfied as well as:
\begin{enumerate}
\item[{\rm 8.}] $a$,$d \in \mathcal{C}^{0}(V)$ and $b$, $c \in \mathcal{C}^{1}(V)$.
\end{enumerate}
\end{The}
\begin{proof}
If $A \in \mathcal{V}(V)$, then there exists $\lambda \in F^{\times}$ such that $A \cdot \gamma(A) = \gamma(A) \cdot A = \lambda \cdot I_{2}$.  Hence, (\ref{un}), (\ref{deux}) and (\ref{trois}) are satisfied.  Moreover,
\begin{equation*}
A^{-1} =  \frac{1}{\lambda} 
\begin{pmatrix}
d^{*} & -b^{*} \\
-c^{*} & a^{*}
\end{pmatrix}
= \frac{1}{\lambda}\gamma(A)
\qquad\hbox{so}\qquad
\alpha(A^{-1}) =  \frac{1}{\lambda}
\begin{pmatrix}
\bar{d} & \bar{b} \\
\bar{c} & \bar{a}
\end{pmatrix}= \frac{1}{\lambda}\beta(A).
\end{equation*}
Since $A \in \mathcal{V}(V)$, we necessarily have $A \cdot B \cdot \alpha(A)^{-1} \in \phi(W)$ for all $B \in \phi(W)$.  Expanding this out in terms of the entries of the matrix $A$ gives (\ref{quatre}), (\ref{cinq}), (\ref{six}) and (\ref{sept}).  The details are left to the reader.
\end{proof}

From now on, we let $H_{2}(V) = \phi(W)$, in other words
\begin{equation*}
H_{2}(V) = \left\{X= \begin{pmatrix}v & \lambda_{1} \\ \lambda_{2} & \bar{v} \end{pmatrix} \middle|\, v \in V \text{ and } \lambda_{1}, \lambda_{2} \in F \right\}.
\end{equation*}  
The $F$-vector space $H_{2}(V)$ has the structure of a non-singular orthogonal geometry coming from $W$ via the isomorphism $\phi$ whose quadratic form $Q$ is given by $Q(X) = v\bar{v} -\lambda_{1}\lambda_{2}$.  We will denote the symmetric $F$-bilinear form on $H_{2}(V)$ by $S$.  Also, given 
\begin{equation*}
A =
\begin{pmatrix}
a & b \\
c & d
\end{pmatrix} \in \mathcal{V}(V),
\qquad\hbox{we let}\qquad
A^{\sharp} = \alpha(A)^{-1} = \frac{1}{\lambda} 
\begin{pmatrix}
\bar{d} & \bar{b}\\
\bar{c} & \bar{a}
\end{pmatrix} = \frac{1}{\lambda}\beta(A),
\end{equation*}
where $\lambda = ad^{*} - bc^{*}$.  The Vahlen group $\mathcal{V}(V)$ acts on $H_{2}(V)$ via
$$A \cdot X = A  X A^{\sharp}, $$
whenever $A \in \mathcal{V}(V)$ and $X \in H_{2}(V)$.  We get a representation 
$$\eta:\mathcal{V}(V) \longrightarrow O(H_{2}(V)).$$  
Since $\phi$ restricted to $W$ gives us an isometry $W \stackrel{\simeq}{\longrightarrow} H_{2}(V)$,
we get an isomorphism of groups 
$$\Phi:O(W) \longrightarrow O(H_{2}(V))$$ 
given by $\Phi(\sigma) = \phi \circ \sigma \circ \phi^{-1}$ for $\sigma \in O(W)$.
\begin{The}
With the notation as above, we have the following commutative diagram
\begin{equation*}
  \begin{CD}
                 1    @>>>  F^{\times}      @>>>   \Gamma(W)       @>{\rho}>> O(W) @>>> 1  \\
                 & & @|   @VV{\phi}V       @VV{\Phi}V   \\
                 1    @>>>   F^{\times}      @>>>        \mathcal{V}(V)       @>{\eta}>> O(H_{2}(V)) @>>> 1 ,
        \end{CD}
\end{equation*}  
where the two vertical maps are isomorphisms of groups and the rows are exact.
\end{The}
\begin{proof}
The proof is simple and left to the reader.
\end{proof}
Similarly, we have the following commutative diagram:
\begin{equation*}
  \begin{CD}
                 1    @>>>  F^{\times}      @>>>   \Gamma^{0}(W)       @>{\rho_{0}}>> SO(W) @>>> 1  \\
                 & & @|   @VV{\phi}V       @VV{\Phi}V   \\
                 1    @>>>   F^{\times}      @>>>        \mathcal{V}^{0}(V)       @>{\eta_{0}}>> SO(H_{2}(V)) @>>> 1 ,
  \end{CD}
\end{equation*}  
whose vertical arrows are isomorphisms, and where $\eta_{0}$ is the restriction of $\eta$ to $\mathcal{V}^{0}(V)$.

One can define a spinor norm for Vahlen groups as follows.  We let $N:\mathcal{V}(V) \longrightarrow F^{\times}$ be defined by $A \mapsto A \cdot \gamma(A)$.  We are now in the setting of \S \ref{abstract}, and we let
$$\mathcal{V}^{+}(V) = Pin^{+}(\eta,N), \qquad\hbox{and}\qquad S\mathcal{V}^{+}(V) = Spin^{+}(\eta_{0},N).$$
It is simple to check that
$$\mathcal{V}^{+}(V) = \phi(Pin^{+}(W)) \qquad\hbox{and}\qquad S\mathcal{V}^{+}(V) = \phi(Spin^{+}(W)).$$
In terms of matrices, we have the following result:
\begin{The} \label{vahlenfurther}
Let 
\begin{equation*}
A =
\begin{pmatrix}
a & b \\
c & d
\end{pmatrix} \in M_{2}(\mathcal{C}(V)).
\end{equation*}
Then,
%\begin{enumerate}
%\item \begin{enumerate}\item $A \in \mathcal{V}_{pin}(V)$ if and only if all conditions of Theorem \ref{vahlenexpand} are satisfied except that $(\ref{un})$ is replaced by:
%$$ad^{*} - bc^{*} = d^{*}a - b^{*}c = \pm 1. $$ 
%\item The matrix $A$ is in $\mathcal{V}_{spin}(V)$ if moreover the following condition is satisfied:
%$$a,d \in \mathcal{C}^{0}(V) \text{ and } b,c \in \mathcal{C}^{1}(V).$$
%\end{enumerate}
\begin{enumerate}
\item $A \in \mathcal{V}^{+}(V)$ if and only if all conditions of Theorem \ref{vahlenexpand} are satisfied with $(\ref{un})$ is replaced by:
$$ad^{*} - bc^{*} = d^{*}a - b^{*}c = 1. $$ 
\item The matrix $A$ is in $S\mathcal{V}^{+}(V)$ if moreover the following condition is satisfied:
$$a,d \in \mathcal{C}^{0}(V) \text{ and } b,c \in \mathcal{C}^{1}(V).$$
\end{enumerate}

\end{The}
\begin{proof}
This should be clear using Lemma \ref{evenodd} and the formula for $A \cdot \gamma(A)$ (which gives the spinor norm for matrices).
\end{proof}

%\subsection{Change of fields}
\subsection{Change of fields} \label{changevahlen}
In this section, we let $F$ be a field of characteristic zero as before, and we let $E$ be a field extension of $F$.  Let $V$ be an orthogonal geometry over $F$, $P$ a hyperbolic plane and set $W = V \perp P$ as before.  Then, the morphism $\psi_{E}:\mathcal{C}(V) \hookrightarrow \mathcal{C}(V_{E})$ of \S \ref{changeoffields} induces an obvious injective morphism of $F$-algebras:
$$\widetilde{\psi}_{E}: M_{2}(\mathcal{C}(V)) \longrightarrow M_{2}(\mathcal{C}(V_{E})). $$
Note that $H_{2}(V_{E})$ is isometric to $H_{2}(V)_{E}$.  Also, we clearly have $\widetilde{\psi}_{E}(\mathcal{V}(V)) \subseteq \mathcal{V}(V_{E})$.  This leads to the following commutative diagram:
\begin{equation} \label{her1}
  \begin{CD}
                 1    @>>>  F^{\times}      @>>>   \mathcal{V}(V)       @>{\eta}>> O({H}_{2}(V)) @>>> 1  \\
                 & & @VVV @VV{\tilde{\psi}_{E}}V          @VVV   \\
                 1    @>>>   E^{\times}      @>>>        \mathcal{V}(V_{E})       @>{\eta_{E}}>> O(H_{2}(V_{E})) @>>> 1 ,
        \end{CD}
\end{equation}  
where the vertical arrows are all injective.  We also have the following important commutative diagram:
\begin{equation} \label{her2}
        \begin{CD}
                1 @>>> \{\pm 1 \} @>>> \mathcal{V}^{+}(V) @>{\eta}>> O(H_{2}(V)) @>{\vartheta}>> F^{\times}/F^{\times 2}  \\
                & & @VVV @VV{\tilde{\psi}_{E}}V @VVV @VVV \\
                1  @>>> \{\pm 1\} @>>> \mathcal{V}^{+}(V_{E}) @>{\eta_{E}}>> O(H_{2}(V_{E})) @>{\vartheta_{E}}>> E^{\times}/E^{\times 2},\\
        \end{CD}
\end{equation}
where only the far right vertical arrow is not injective.  One has a similar diagram for $S\mathcal{V}^{+}(V)$.

%\section{Generalized Cartan matrices, system of simple roots and Weyl groups}
\section{Generalized Cartan matrices, system of simple roots and Weyl groups} \label{cartanmatrices}
Throughout this section, $F$ will be a field of characteristic zero; hence, in particular $\mathbb{Q} \subseteq F$.  Most of the proofs of this section will be omited, since they are standard or can be easily provided.

We remind the reader of the following definitions (see for instance page $1$ of \cite{Ka90}):
\begin{Def}
An $n \times n$ matrix $C = (c_{ij})$ is called a (generalized) Cartan matrix if it satisfies the following conditions:
\begin{enumerate}
\item $c_{ii} = 2$ for all $i = 1,\ldots,n$,
\item $c_{ij} \in \mathbb{Z}_{\le 0}$ for all $i,j$ satisfying $i \neq j$,
\item $c_{ij} = 0$ if and only if $c_{ji} = 0$.
\end{enumerate}
A Cartan matrix $C$ is called non-singular if it has full rank.  
\end{Def}

In order to study Weyl groups, we first introduce the notion of a system of simple roots.  Throughout this section, $V$ will be a finite dimensional orthogonal geometry over $F$ with a symmetric bilinear form $S$.  If $v \in V$ is a non-isotropic vector and $w \in W$ is arbitrary, then we defined the Cartan bracket $\langle v,w \rangle$ as usual by the formula
$$\langle v , w \rangle = \frac{2S(v,w)}{S(v,v)}.$$
\begin{Def}
A system of simple roots in $V$ consists of a basis $\Pi = \{\alpha_{1},\ldots,\alpha_{n} \}$ of $V$ such that
\begin{enumerate}
\item The elements of $\Pi$ are non-isotropic vectors,
\item The matrix $(\langle \alpha_{i},\alpha_{j}\rangle)$ is a Cartan matrix.
\end{enumerate}
A system of simple roots $(V,\Pi)$ is called non-singular if $V$ is a non-singular orthogonal geometry.
\end{Def}
A system of simple roots will typically be denoted by $(V,\Pi)$.  
\begin{Def} \label{sym}
A Cartan matrix $C$ is called symmetrizable over $F$ if $D \cdot C = B$, where $B,D \in M_{n}(F)$ are such that
\begin{enumerate}
\item $B$ is symmetric,
\item $D = {\rm diag}(\varepsilon_{1},\ldots,\varepsilon_{n})$ is diagonal,
\item ${\rm det}(D) \neq 0$.
\end{enumerate}
\end{Def}

\begin{Lem}
Let $(V,\Pi)$ be a system of simple roots with associated Cartan matrix $C$.  Then $C$ is symmetrizable.  Moreover, $C$ is non-singular if and only if $(V,\Pi)$ is non-singular.
\end{Lem}
Given any symmetrizable Cartan matrix $C$, there always exists a system of simple roots $(V,\Pi)$ with associated Cartan matrix $C$.  Indeed, let $V$ be a vector space of dimension $n$ over $F$, and let $\Pi = \{\alpha_{1},\ldots,\alpha_{n} \}$ be any basis for $V$.  Since $C$ is assumed to be symmetrizable, we can write 
\begin{equation} \label{symdec}
D \cdot C = B,
\end{equation} 
where $D = {\rm diag}(\varepsilon_{1},\ldots,\varepsilon_{n})$ is a diagonal matrix, $B$ a symmetric matrix and ${\rm det}(D) \neq 0$.  One can define a symmetric $F$-bilinear form $\kappa$ on $V$ via the formula $\kappa(\alpha_{i},\alpha_{j}) = b_{ij}$.  We then have an orthogonal geometry on $V$.  Now, one has $\kappa(\alpha_{i},\alpha_{i}) = 2\varepsilon_{i}$ and $\kappa(\alpha_{i},\alpha_{j}) = \varepsilon_{i}c_{ij}$.  It follows that the entries of the Cartan matrix $C$ satisfy
$$c_{ij} = \frac{2 \cdot \kappa(\alpha_{i},\alpha_{j})}{\kappa(\alpha_{i},\alpha_{i})} = \langle\alpha_{i},\alpha_{j} \rangle. $$
Therefore, $(V,\Pi)$ is a system of simple roots with associated Cartan matrix $C$.

\begin{Def}
A Cartan matrix $C = (c_{ij})$ is called reducible if there exists a permutation $\tau \in S_{n}$ such that $(c_{\tau(i) \tau(j)})$ is block diagonal with more than one block.  Otherwise, it is called irreducible.  
\end{Def}
Note that if $(V,\Pi)$ is a system of simple roots and $\Pi_{1} \subseteq \Pi$ is a non-empty subset, then $(W,\Pi_{1})$, where $W = {\rm Span}(\Pi_{1})$ is also a system of simple roots.
\begin{Def}
A system of simple roots $(V,\Pi)$ is called reducible if there exist $\Pi_{1},\Pi_{2} \subseteq \Pi$ satisfying the following conditions:
\begin{enumerate}
\item $\Pi_{1},\Pi_{2} \neq \varnothing$,
\item $\Pi = \Pi_{1} \sqcup \Pi_{2}$ (disjoint union),
\item $V = V_{1} \perp V_{2}$, where $V_{i} = Span(\Pi_{i})$, for $i=1,2$.
\end{enumerate}
Otherwise, $(V,\Pi)$ is called irreducible.
\end{Def}
It is clear that a system $(V,\Pi)$ of simple roots is irreducible if and only if the corresponding Cartan matrix is irreducible.
\begin{Pro}
Given any system of simple roots $(V,\Pi)$, one can find non-empty subsets $\Pi_{1},\ldots,\Pi_{s} \subseteq \Pi$ satisfying:
\begin{enumerate}
\item $\Pi = \Pi_{1} \sqcup \ldots \sqcup \Pi_{s}$,
\item $V = V_{1} \perp \ldots \perp V_{s}$, where $V_{i} = {\rm Span}(\Pi_{i})$, for $i=1,\ldots,s$,
\item $(V_{i},\Pi_{i})$ are irreducible systems of simple roots.
\end{enumerate}
\end{Pro}
Because of this last proposition, one can focus on irreducible Cartan matrices or equivalently on irreducible systems of simple roots.
\begin{Lem}
If $(V,\Pi)$ is an irreducible system of simple roots and $F$ a totally ordered field, then all the $S(\alpha_{i},\alpha_{i})$ have the same sign.
\end{Lem}
Now, we introduce the notion of morphism between systems of simple roots.
\begin{Def}
A morphism between two systems of simple roots $(V_{1},\Pi_{1})$ and $(V_{2},\Pi_{2})$ is an $F$-linear transformation $\phi:V_{1} \longrightarrow V_{2}$ such that
\begin{enumerate}
\item $\phi(\Pi_{1}) \subseteq \Pi_{2}$,
\item $\langle \phi(\alpha),\phi(\beta) \rangle = \langle \alpha, \beta \rangle$ for all $\alpha, \beta \in \Pi_{1}$.
\end{enumerate}
\end{Def}
\begin{The} \label{morphismssystem}
If $\phi:(V_{1},\Pi_{1}) \longrightarrow (V_{2},\Pi_{2})$ is a morphism of systems of simple roots and $(V_{1},\Pi_{1})$ is irreducible, then $\phi$ is an orthogonal similitude.
\end{The}

\begin{Lem} \label{symfactor}
Let $C$ be an irreducible symmetrizable Cartan matrix and assume that we are given two different symmetrizations $D_{1} \cdot C = B_{1}$ and $D_{2} \cdot C = B_{2}$, where the $D_{i}$ are non-singular and diagonal, and the $B_{i}$ are symmetric matrices.  Then, there exists $\lambda \in F^{\times}$ such that $\lambda D_{1} = D_{2}$.
\end{Lem}

\begin{Def}
Given $(V,\Pi)$, we define the Weyl group to be $\cW(\Pi) = \langle r_{\alpha} \, | \, \alpha \in \Pi  \rangle$.
\end{Def}

If $\phi:(V_{1},\Pi_{1}) \longrightarrow (V_{2},\Pi_{2})$ is an isomorphism of irreducible systems of simple roots, then it is simple to check that $\phi$ induces an isomorphism of groups $\tilde{\phi}:O(V_{1}) \stackrel{\simeq}{\longrightarrow} O(V_{2})$ defined by $\sigma \mapsto \tilde{\phi}(\sigma) = \phi \circ \sigma \circ \phi^{-1}$.
\begin{Lem}
If $\phi:(V_{1},\Pi_{1}) \longrightarrow (V_{2},\Pi_{2})$ is an isomorphism of irreducible systems of simple roots, then given any non-isotropic vector $v_{1} \in V_{1}$, the vector $\phi(v_{1})$ is also non-isotropic, and moreover
$$\tilde{\phi}(r_{v_{1}}) = r_{\phi(v_{1})}.$$
\end{Lem}
As a corollary, we obtain
\begin{Cor}
If $\phi:(V_{1},\Pi_{1}) \longrightarrow (V_{2},\Pi_{2})$ is an isomorphism of irreducible systems of simple roots, then the group isomorphism $\tilde{\phi}$ of above induces an isomorphism $\cW(\Pi_{1}) \stackrel{\simeq}{\longrightarrow} \cW(\Pi_{2})$.
\end{Cor}

In Chapter $4$ of \cite{Ka90}, Kac classifies the irreducible Cartan matrices in three distinct classes:  finite type, affine type, and indefinite type.  They correspond to finite-dimensional simple Lie algebras, affine Kac-Moody algebras and indefinite Kac-Moody algebras respectively.  The finite and affine type ones have been classified, and one can find their corresponding Dynkin diagrams in Chapter $4$ of \cite{Ka90}.  It is known that all finite and affine Cartan matrices are symmetrizable.  Among the indefinite Kac-Moody algebras, the so-called hyperbolic ones have been the most extensively studied.  
\begin{Def}
An irreducible Cartan matrix is called hyperbolic if it is symmetrizable of indefinite type and if every proper connected subdiagram of its corresponding Dynkin diagram is of finite or affine type.  The corresponding Kac-Moody algebras are called hyperbolic Kac-Moody algebras.
\end{Def}
The hyperbolic Cartan matrices of rank greater than or equal to $3$ and their corresponding Dynkin diagrams have also been classified.  It is known that there are no hyperbolic Cartan matrices of rank strictly greater than $10$.  See for instance \cite{Carbone:2010}.  For the purpose of this paper, we also introduce the following non-standard terminology.  If $C$ is an irreducible non-singular symmetrizable Cartan matrix, then one can write $D \cdot C = B$ for some non-singular matrices $D,B \in M_{n}(\mathbb{R})$ such that $D$ is diagonal and $B$ is symmetric with signature $(p,q)$.  Then, the number $\iota(C) := |p-q|$ does not depend on the choice of the decomposition $D \cdot C = B$ by Lemma \ref{symfactor}.
\begin{Def}
An irreducible non-singular symmetrizable Cartan matrix is called Lorentzian if $\iota(C) = 1$.
\end{Def}
We warn the reader that this is our own terminology, and the word Lorentzian can mean something else in the literature.  The corresponding Weyl groups and system of simple roots will be called of finite, affine, hyperbolic and Lorentzian type respectively.  Note that if $(V,\Pi)$ is a Lorentzian system of simple roots over $\mathbb{R}$, the Weyl group 
$\cW(\Pi)$ can be viewed as a subgroup of $O(\mathbb{R}^{n-1,1}) = O(\mathbb{R}^{1,n-1})$.

%\subsection{A useful normalization}
\subsection{A useful normalization} \label{normalization}
If $C$ is an irreducible symmetrizable Cartan matrix, then one can write $D \cdot C = B$ for some non-singular diagonal matrix $D$ and where $B$ is symmetric.  This decomposition is not unique, as Lemma \ref{symfactor} shows.  On the other hand, it follows from Lemma \ref{symfactor} that there is a unique matrix $D = {\rm diag}(\varepsilon_{1},\ldots,\varepsilon_{n})$ satisfying
\begin{enumerate}
\item $D$ is non-singular,
\item $\varepsilon_{i} \in \mathbb{Z}_{>0}$,
\item ${\rm gcd}(\varepsilon_{1},\ldots,\varepsilon_{n}) = 1$,
\item $D \cdot C$ is a symmetric matrix.  
\end{enumerate}
Note that $D \cdot C \in M_{n}(\mathbb{Z})$.  \emph{From now on, given an irreducible symmetrizable Cartan matrix $C$, we will always take the representation $D \cdot C = B$, where $D$ is normalized as above.}  (Note that if $C$ is symmetric, then $D$ is the identity matrix.)

%\subsection{Canonical Lorentzian extensions}
\subsection{Canonical Lorentzian extensions}
As explained on page $71$ of \cite{Ka90}, starting with an irreducible system of simple roots of finite type, one can extend it to a Lorentzian system of simple roots.  Kac calls the resulting systems of simple roots ``canonical hyperbolic extensions''.  Unfortunately, they are not always hyperbolic systems of simple roots.  We shall rather refer to them as ``canonical Lorentzian extensions'' or ``canonical double extensions'', to avoid any ambiguity.  The purpose of this section is to recall how this is done.  

One starts with a Cartan matrix $C$ of finite type $T_{n}$ (where $T_{n} = A_{n} (n \ge 1),B_{n} (n \ge 2),C_{n} (n \ge 3),D_{n} (n \ge 4),E_{6},E_{7},E_{8},F_{4}$ or $G_{2}$).  The Cartan matrices of finite type are known to be symmetrizable; hence, we let $B = D \cdot C \in M_{n}(\mathbb{Z})$, where $D$ is the unique diagonal matrix normalized as explained in \S \ref{normalization}.  We let $V = F^{n}$ and we define an orthogonal geometry over $F$ via $\kappa(\alpha_{i},\alpha_{j})=b_{ij}$, where $\Phi = (\alpha_{1},\ldots,\alpha_{n})$ is the standard basis for $V$.  Then $(V,\Phi)$ is a system of simple roots with Cartan matrix $C$.  Letting $\theta$ be the highest root, we let $m = \kappa(\theta,\theta)$, and we set
$$W = V^{\frac{1}{m}} \perp P, $$
where $P$ is a hyperbolic plane.  We also fix a hyperbolic pair $(f_{1},f_{2})$ for $P$.  For the convenience of the reader, we list below the highest root for each type as well as the integer $m = \kappa(\theta,\theta)$.
\begin{align*}
A_{n} (n \ge 1) &&& \theta = \alpha_{1} + \ldots + \alpha_{n} && 2\\
B_{n} (n \ge 2) &&& \theta = \alpha_{1} + 2 \alpha_{2} + \ldots + 2 \alpha_{n} && 4\\
C_{n} (n \ge 3) &&& \theta = 2 \alpha_{1} + \ldots + 2\alpha_{n-1} + \alpha_{n} && 4\\
D_{n} (n \ge 4) &&& \theta = \alpha_{1} + 2\alpha_{2} + \ldots + 2 \alpha_{n-2} + \alpha_{n-1} + \alpha_{n} && 2\\
G_{2} &&& \theta = 2\alpha_{1} + 3\alpha_{2} && 6\\
F_{4} &&& \theta = 2 \alpha_{1} + 3 \alpha_{2} + 4 \alpha_{3} + 2 \alpha_{4} && 4\\
E_{6} &&& \theta = \alpha_{1} + 2 \alpha_{2} + 3\alpha_{3} + 2 \alpha_{4} + \alpha_{5} + 2\alpha_{6} && 2\\
E_{7} &&& \theta = 2 \alpha_{1} + 3\alpha_{2} + 4 \alpha_{3} + 3 \alpha_{4} + 2 \alpha_{5} + \alpha_{6} + 2 \alpha_{7} && 2\\
E_{8} &&& \theta = 2 \alpha_{1} + 3 \alpha_{2} + 4 \alpha_{3} + 5 \alpha_{4} + 6 \alpha_{5} + 4 \alpha_{6} + 2 \alpha_{7} + 3 \alpha_{8} && 2
\end{align*}
Setting
$$\alpha_{-1} = f_{1} - f_{2}\qquad\hbox{and}\qquad \alpha_{0} = -f_{1} - \theta, $$
and $\Pi = (\alpha_{-1},\alpha_{0},\alpha_{1},\ldots,\alpha_{n})$, one checks that $(W,\Pi)$ is a Lorentzian system of simple roots for a certain Cartan matrix which we will denote by $C^{++}$.  The corresponding system of simple roots will be said to be of type $T_{n}^{++}$.  

It turns out that the Dynkin diagram corresponding to $C^{++}$ can be obtained from the Dynkin diagram corresponding to the affine type root system $T_{n}^{(1)}$ after adding one vertex labelled $\alpha_{-1}$ connected to $\alpha_{0}$ by a single edge.  Again, for the convenience of the reader, we list below the Dynkin diagrams corresponding to $T_{n}^{(1)}$.  
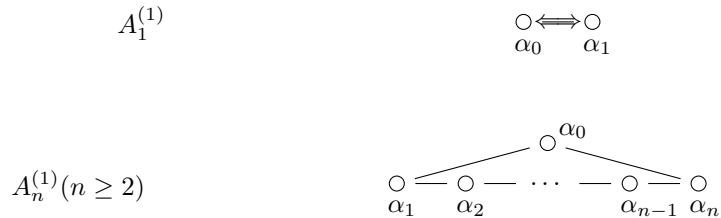
\begin{align*}
A_1^{(1)} &&&
\begin{tikzpicture}[start chain]
\dnode{0}
\dnode{1}
\path (chain-1) -- node[anchor=mid] {\(\Longleftrightarrow\)} (chain-2);
\end{tikzpicture}
\end{align*}

\begin{align*}
A_n^{(1)} (n \ge 2) &&&
\begin{tikzpicture}[start chain,node distance=1ex and 2em]
\dnode{1}
\dnode{2}
\dydots
\dnode{n-1}
\dnode{n}
\begin{scope}[start chain=br going above]
\chainin(chain-3);
\node[ch,join=with chain-1,join=with chain-5,label={[inner sep=1pt]10:\(\alpha_{0}\)}] {};
\end{scope}
\end{tikzpicture}
\end{align*}

\begin{align*}
B_n^{(1)} (n \ge 3) &&&
\begin{tikzpicture}
\begin{scope}[start chain]
\dnode{0}
\dnode{2}
\dnode{3}
\dydots
\dnode{n-1}
\dnodenj{n}
\end{scope}
\begin{scope}[start chain=br going above]
\chainin(chain-2);
\dnodebr{1}
\end{scope}
\path (chain-5) -- node{\(\Rightarrow\)} (chain-6);
\end{tikzpicture}
\end{align*}

\begin{align*}
C_n^{(1)} (n \ge 2) &&&
\begin{tikzpicture}[start chain]
\dnodenj{0}
\dnodenj{1}
\dydots
\dnode{n-1}
\dnodenj{n}
\path (chain-1) -- node{\(\Rightarrow\)} (chain-2);
\path (chain-4) -- node{\(\Leftarrow\)} (chain-5);
\end{tikzpicture}
\end{align*}

\begin{align*}
D_n^{(1)} (n \ge 4) &&&
\begin{tikzpicture}
\begin{scope}[start chain]
\dnode{0}
\dnode{2}
\dnode{3}
\dydots
\dnode{n-2}
\dnode{n}
\end{scope}
\begin{scope}[start chain=br going above]
\chainin(chain-2);
\dnodebr{1};
\end{scope}
\begin{scope}[start chain=br2 going above]
\chainin(chain-5);
\dnodebr{n-1};
\end{scope}
\end{tikzpicture}
\end{align*}

\begin{align*}
G_2^{(1)} &&&
\begin{tikzpicture}[start chain]
\dnode{0}
\dnode{1}
\dnodenj{2}
\path (chain-2) -- node{\(\Rrightarrow\)} (chain-3);
\end{tikzpicture} 
\end{align*}

\begin{align*}
F_4^{(1)} &&&
\begin{tikzpicture}[start chain]
\dnode{0}
\dnode{1}
\dnode{2}
\dnodenj{3}
\dnode{4}
\path (chain-3) -- node[anchor=mid]{\(\Rightarrow\)} (chain-4);
\end{tikzpicture} 
\end{align*}
\begin{align*}
E_6^{(1)} &&&
\begin{tikzpicture}
\begin{scope}[start chain]
\foreach \dyi in {0,6,3,4,5} {
\dnode{\dyi}
}
\end{scope}
\begin{scope}[start chain=br going above]
\chainin(chain-3);
\dnodebr{2}
\dnodebr{1}
\end{scope}
\end{tikzpicture} 
\end{align*}

\begin{align*}
E_7^{(1)} &&&
\begin{tikzpicture}
\begin{scope}[start chain]
\foreach \dyi in {0,1,2,3,4,5,6} {
\dnode{\dyi}
}
\end{scope}
\begin{scope}[start chain=br going above]
\chainin(chain-4);
\dnodebr{7}
\end{scope}
\end{tikzpicture} 
\end{align*}

\begin{align*}
E_8^{(1)} &&&
\begin{tikzpicture}
\begin{scope}[start chain]
\foreach \dyi in {0,1,2,3,4,5,6,7} {
\dnode{\dyi}
}
\end{scope}
\begin{scope}[start chain=br going above]
\chainin(chain-6);
\dnodebr{8}
\end{scope}
\end{tikzpicture}
\end{align*}
As an example, one recovers from the table that the Cartan matrix corresponding to the canonical Lorentzian extension $B_{3}^{++}$ is given by
\begin{equation*}
\begin{pmatrix}
2 & -1 & 0 & 0 & 0 \\
-1 & 2 & 0 & -1 & 0 \\
0 & 0 & 2 & -1 & 0 \\
0 & -1 & -1 & 2 & -1 \\
0 & 0 & 0 & -2 & 2
\end{pmatrix}
\end{equation*}

Among the canonical Lorentzian extensions $T_{n}^{++}$, the following ones are hyperbolic:  $A_{n}^{++}$ for $n=1,2,3,4,5,6,7$, $B_{n}^{++}$ for $n=3,4,5,6,7,8$, $C_{n}^{++}$ for $n=2,3,4$, $D_{n}^{++}$ for $n=4,5,6,7,8$, $E_{n}^{++}$ for $n=6,7,8$, $F_{4}^{++}$, and $G_{2}^{++}$.  Among the hyperbolic Lorentzian canonical extensions $T_{n}^{++}$, the following ones have a symmetric Cartan matrix:  $A_{n}^{++}$ for $n=1,2,3,4,5,6,7$, $D_{n}^{++}$ for $n=4,5,6,7,8$, and $E_{n}^{++}$ for $n=6,7,8$.

%\subsection{Change of fields}
\subsection{Change of fields}
As before, $F$ is a field of characteristic zero and $E/F$ is a field extension.  Throughout this section, $V$ will be a non-singular geometry.  Our goal in this section is to show that the Weyl group is preserved under base change, and this is the content of Proposition \ref{Weylpreservedunderbasechange} below.  If $V$ is an $F$-vector space, then we let $V_{E} = E \otimes_{F} V$.  

Suppose that $(V,\Pi)$ is a system of simple roots, where $\Pi = \{\alpha_{1},\ldots,\alpha_{n}\}$.  Define $\alpha_{i,E} \in V_{E}$ by $\alpha_{i,E} = 1 \otimes \alpha_{i}$ and let $\Pi_{E} = \{\alpha_{1,E},\ldots,\alpha_{n,E} \}$.  It is then simple to check that $(V_{E},\Pi_{E})$ is also a system of simple roots with associated Cartan matrix $C$.  We have an injective group morphism $\iota_{E}:O(V) \hookrightarrow O(V_{E})$.  If we let $r_{i,E}$ denote the simple reflection associated to $\alpha_{i,E}$ and $r_{i}$ the one associated to $\alpha_{i}$, then it is simple to check that $\iota_{E}(r_{i}) = r_{i,E}$.  Hence, we obtain the following result.
\begin{Pro} \label{Weylpreservedunderbasechange}
With the notation as above, we have $\iota_{E}(\cW(\Pi)) = \cW(\Pi_{E})$.
\end{Pro}

%\section{Weyl groups of the hyperbolic Lorentzian extensions $T_{n}^{++}$ with symmetric matrices}
\section{Weyl groups of the hyperbolic canonical Lorentzian extensions $T_{n}^{++}$ with symmetric Cartan matrices} \label{symmetric}
In this section, we let $F=\mathbb{Q}$, $V = \mathbb{Q}^{n}$ and we let $(\alpha_{1},\ldots,\alpha_{n})$ be the standard basis for $V$.  Starting with a Cartan matrix $C$ of finite type, we endow $V$ with an orthogonal geometry as follows.  We consider $B = D \cdot C$, where $D$ is the unique diagonal matrix normalized as in \S \ref{normalization}.  The orthogonal geometry on $V$ is given by $\kappa(\alpha_{i},\alpha_{j})=b_{ij}$.  We let 
$$\Lambda = \sum_{i = 1}^{n}\mathbb{Z} \cdot \alpha_{i} \subseteq V $$
be the corresponding root lattice.  Note that because of our choice of $D$, we have $\kappa(\Lambda,\Lambda) \subseteq \mathbb{Z}$.  If $\theta$ is the highest root, then we let $m = \kappa(\theta,\theta)$ and we consider $W = V^{\frac{1}{m}} \perp P$, where $P$ is a hyperbolic plane for which we fix a hyperbolic pair $(f_{1},f_{2})$.  We then have the universal Clifford algebra $\mathcal{C}(W)$ and an isomorphism of $\mathbb{Q}$-algebras $\Phi:\mathcal{C}(W) \stackrel{\simeq}{\longrightarrow} M_{2}(\mathcal{C}(V^{\frac{1}{m}}))$ induced by the $\mathbb{Q}$-linear transformation $W \longrightarrow M_{2}(\mathcal{C}(V^{\frac{1}{m}}))$ defined by
\begin{equation*}
v + \lambda_{1}f_{1} + \lambda_{2}f_{2} \mapsto  \begin{pmatrix} v & \lambda_{1} \\ \lambda_{2} & \bar{v}\end{pmatrix}.
\end{equation*}
As before, we let 
$$H_{2}(V^{\frac{1}{m}}) = \left\{X = \begin{pmatrix} v & \lambda_{1} \\ \lambda_{2} & \bar{v} \end{pmatrix} \middle| \, v \in V^{\frac{1}{m}} \text{ and } \lambda_{1}, \lambda_{2} \in \mathbb{Q} \right\}\subseteq M_{2}(\mathcal{C}(V^{\frac{1}{m}})). $$
The $\mathbb{Q}$-vector space $H_{2}(V^{\frac{1}{m}})$ has dimension $ n+ 2$ and its signature is $(n+1,1)$.  We recall that $H_{2}(V^{\frac{1}{m}})$ is an orthogonal geometry with quadratic form $Q$ given by $Q(X) = v \bar{v} - \lambda_{1}\lambda_{2}$.  Moreover, the group $\mathcal{V}(V^{\frac{1}{m}})$ acts on $H_{2}(V^{\frac{1}{m}})$ via $\eta(A)(X) =  AXA^{\sharp}$, and this gives the short exact sequence
$$1 \longrightarrow \mathbb{Q}^{\times} \longrightarrow \mathcal{V}(V^{\frac{1}{m}}) \stackrel{\eta}{\longrightarrow} O(H_{2}(V^{\frac{1}{m}})) \longrightarrow 1. $$
Moreover, if $X \in H_{2}(V^{\frac{1}{m}})$ is a non-isotropic vector, then
\begin{equation} \label{niceref}
\eta(X) = r_{X}. 
\end{equation}
The simple roots $\alpha_{i}$ ($i=-1,0,1,\ldots,n$) of the canonical Lorentzian extension are mapped to $X_{i} \in H_{2}(V^{\frac{1}{m}})$, where
$$X_{-1} = \begin{pmatrix} 0 & 1 \\  -1 & 0  \end{pmatrix} , \, \, \,X_{0} = \begin{pmatrix}-\theta & -1 \\ 0 & -\bar{\theta}  \end{pmatrix}\ \ \hbox{and}\ \  X_{i} = \begin{pmatrix} \alpha_{i} & 0 \\ 0 & \bar{\alpha}_{i}  \end{pmatrix}\ \ 
\hbox{for}\ \ i=1,\ldots,n.$$
The root lattice of this system of simple roots will be denoted by
$$\Lambda^{++} = \sum_{i=-1}^{n} \mathbb{Z} X_{i} \subseteq H_{2}(V^{\frac{1}{m}}). $$
\begin{Lem}
With the notation as above, one has $\Lambda^{++} = H_{2}(\Lambda)$, where
$$H_{2}(\Lambda) = \left\{\begin{pmatrix}v & n_{1}  \\ n_{2} & \bar{v}  \end{pmatrix} \middle| \, v \in \Lambda \text{ and } n_{1}, n_{2} \in \mathbb{Z} \right\} \subseteq M_{2}(\mathcal{C}(V^{\frac{1}{m}})). $$
\end{Lem}
\begin{proof}
The inclusion $\Lambda^{++} \subseteq H_{2}(\Lambda)$ is clear.  The other inclusion follows from the identity
\begin{equation*}
\begin{pmatrix}
0 & n_{1}  \\  n_{2} & 0
\end{pmatrix}
= -(n_{1} + n_{2})X_{0} - n_{2}X_{-1} - (n_{1} + n_{2}) \begin{pmatrix} \theta & 0 \\ 0 & \bar{\theta}  \end{pmatrix} \in \Lambda^{++}.
\end{equation*}
\end{proof}
We let
$$O(H_{2}(\Lambda)) = \left\{\sigma \in O(H_{2}(V^{\frac{1}{m}})) \, | \, \sigma\left(H_{2}(\Lambda) \right) = H_{2}(\Lambda) \right\}, $$
be the group of units of this lattice.  We also let
$$O^{+}(H_{2}(\Lambda)) = O(H_{2}(\Lambda)) \cap O^{+}(H_{2}(V^{\frac{1}{m}})). $$
The groups $SO(H_{2}(\Lambda))$ and $SO^{+}(H_{2}(\Lambda))$ are defined similarly.

From now on, we let $\mathcal{O}$ be the $\mathbb{Z}$-algebra generated by the $\alpha_{i}$, that is $\mathcal{O} = \mathbb{Z}[\alpha_{1},\ldots,\alpha_{n}] \subseteq \mathcal{C}(V^{\frac{1}{m}})$.  Then, $\mathcal{O}$ is an order in the finite dimensional $\mathbb{Q}$-algebra $\mathcal{C}(V^{\frac{1}{m}})$.  Note that $\mathcal{O}$ is closed under the (anti) involutions $', ^{*}$, and $\bar{}\,$.  
\begin{Lem} \label{inters}
The notation being as above and assuming that $m=2$, a $\mathbb{Z}$-basis for $\mathcal{O}$ is given by the elements $\{\alpha_{I} \, | \, I \in \Omega \}$.
\end{Lem}
\begin{proof}
Note that the elements $\{\alpha_{I} \, | \, I \in \Omega \}$ form a basis for the $\mathbb{Q}$-vector space $\mathcal{C}(V^{\frac{1}{m}})$.  In fact, if $m=2$ these elements form a $\mathbb{Z}$-basis for $\mathcal{O}$.  Indeed, it suffices to show that they generate $\mathcal{O}$ as a $\mathbb{Z}$-module, but this follows from the relation
$$\alpha_{i} \alpha_{j} + \alpha_{j} \alpha_{i} = -2\frac{1}{m}\kappa(\alpha_{i},\alpha_{j}) = - \kappa(\alpha_{i},\alpha_{j}), $$
and noting that $\kappa(\Lambda,\Lambda) \subseteq \mathbb{Z}$.  
\end{proof}
\emph{For the remainder of \S \ref{symmetric}, we assume that $m=2$ so that we are dealing with the simply laced canonical Lorentzian extensions (meaning that they have a symmetric Cartan matrix).}  We then set
$$\mathcal{V}(\mathcal{O}) = \mathcal{V}(V^{\frac{1}{2}}) \cap M_{2}(\mathcal{O})^{\times} \qquad\hbox{and}\qquad\mathcal{V}^{+}(\mathcal{O}) = \mathcal{V}^{+}(V^{\frac{1}{2}}) \cap M_{2}(\mathcal{O})^{\times}. $$
We define similarly the groups $S\mathcal{V}(\mathcal{O})$ and $S\mathcal{V}^{+}(\mathcal{O})$.
\begin{The} \label{vahlenexpandO}
Let 
\begin{equation*}
A =
\begin{pmatrix}
a & b \\
c & d
\end{pmatrix} \in M_{2}(\mathcal{O}).
\end{equation*}
Then $A \in \mathcal{V}(\mathcal{O})$ if and only if the following conditions are satisfied.
\begin{enumerate}
\item $ad^{*} - bc^{*} = d^{*}a - b^{*}c = \pm 1$, \label{1un}
\item $ba^{*} - ab^{*} = cd^{*} - dc^{*} = 0$, \label{2deux}
\item $a^{*}c - c^{*}a = d^{*}b - b^{*}d = 0$, \label{3trois}
\item $a \bar{a}, b \bar{b}, c \bar{c}, d \bar{d} \in \mathbb{Z}$, \label{4quatre}
\item $b \bar{d}, a \bar{c} \in \Lambda$, \label{5cinq}
\item $av\bar{b} + b\bar{v}\bar{a}, cv\bar{d} + d\bar{v}\bar{c} \in \mathbb{Z}$ for all $v \in \Lambda$, \label{6six}
\item $av\bar{d} + b\bar{v} \bar{c} \in \Lambda$ for all $v \in \Lambda$. \label{7sept}
\end{enumerate}
\end{The}
\begin{proof}
When $m=2$, it follows from Lemma \ref{inters} that $\mathbb{Q} \cap \mathcal{O} = \mathbb{Z}$ and $V \cap \mathcal{O} = \Lambda$.  The rest of the proof is similar to the proof of Theorem \ref{vahlenexpand}, and the details are left to the reader.
\end{proof}

From this last theorem, we obtain the following result.
\begin{Pro}
With the notation being as above, one has
\begin{enumerate}
\item $\eta\left(\mathcal{V}(\mathcal{O}) \right) \subseteq O(H_{2}(\Lambda))$,
\item $\eta\left(\mathcal{V}^{+}(\mathcal{O}) \right) \subseteq O^{+}(H_{2}(\Lambda))$
\item $\eta\left(S\mathcal{V}(\mathcal{O}) \right) \subseteq SO(H_{2}(\Lambda))$
\item $\eta\left(S\mathcal{V}^{+}(\mathcal{O}) \right) \subseteq SO^{+}(H_{2}(\Lambda))$
\end{enumerate}
\end{Pro}
\begin{proof}
This follows directly from Theorem \ref{vahlenexpandO}.  The details are left to the reader.
\end{proof}

Since we are assuming that $m=2$,  we have $Q(X_{i})=1$ for all $i=-1,0,1,\ldots,n$.  Therefore, the corresponding Weyl group, which we denote here by $\cW(C^{++})$, is a subgroup of 
$O^{+}(V^{\frac{1}{2}})$.  Letting
$$\Gamma = \langle X_{i} \, | \, i=-1,0,\ldots,m \rangle \le \mathcal{V}^{+}(\mathcal{O}), $$
we have $\eta(\Gamma) = \cW(C^{++})$ because of (\ref{niceref}).  Therefore, we have the following chain of subgroups:
\begin{equation} \label{fundincl}
\cW(C^{++}) \subseteq \eta\left( \mathcal{V}^{+}(\mathcal{O})\right) \subseteq O^{+}(H_{2}(\Lambda)).
\end{equation}
From now on, we let $\cP\mathcal{V}(\mathcal{O}) = \mathcal{V}(\mathcal{O})/\{\pm 1 \}$, and we define similarly $\cP\mathcal{V}^{+}(\mathcal{O})$, $\cP S\mathcal{V}(\mathcal{O})$, and $\cP S\mathcal{V}^{+}(\mathcal{O})$.

\begin{The} \label{main}
With the notation as above, one has that $\eta$ induces an isomorphism of groups
$$\eta:\cP\mathcal{V}^{+}(\mathcal{O}) \stackrel{\simeq}{\longrightarrow} \cW(C^{++}) $$
for the following hyperbolic canonical Lorentzian extensions:
\begin{enumerate}
\item $A_{n}^{++}$ for $n=1,2,3,4,5,6$,
\item $D_{n}^{++}$ for $n=5,6,7,8$,
\item $E_{n}^{++}$ for $n=6,7,8$.
\end{enumerate}
\end{The}
\begin{proof}
Corollary $5.10$ of \cite{Ka90} shows that for each hyperbolic canonical Lorentzian extension with symmetric Cartan matrix, one has
$$O(H_{2}(\Lambda)) = \pm {\rm Aut}(C^{++}) \ltimes \cW(C^{++}), $$
where ${\rm Aut}(C^{++})$ is the group of outer automorphisms of the corresponding Dynkin diagram.  It is not hard
to see that $-id \notin O^{+}(H_{2}(\Lambda))$.  For each hyperbolic canonical Lorentzian extension with a symmetric Cartan matrix, we computed the spinor norm of all the automorphism $\pm a$, where $a$ is an outer automorphism.  It turns out that these spinor norms are non-trivial exactly in the cases listed in the theorem.  Thus, in those cases, the chain of subgroups (\ref{fundincl}) induces the following equalities
$$\cW(C^{++}) = \eta\left( \mathcal{V}^{+}(\mathcal{O})\right) = O^{+}(H_{2}(\Lambda)).$$
Our result follows from this last equality.  We did not find a conceptual way of computing the spinor norm of the outer automorphisms, so we used the software PARI (\cite{PARI}) to calculate those.  We explain these calculations in \S \ref{spinorouter} below.
\end{proof}

%\subsection{Spinor norm of the outer autormorphism}
\subsection{Spinor norm of outer autormorphisms} \label{spinorouter}
Given an automorphism of a Dynkin diagram $\sigma$, it induces an automorphism of the root lattice via the formula $\sigma(\alpha_{i}) = \alpha_{\sigma(i)}$.  For example, the permutation $(1 \, 4)(2 \, 3)$ of the Dynkin diagram corresponding to $A_{4}^{++}$ is such an outer automorphism.  For each hyperbolic canonical Lorentzian extension with a symmetric Cartan matrix, we calculated the spinor norm of the automorphisms $\pm a$, where $a$ runs through the outer automorphisms, using PARI.  Each time, we found an orthogonal basis for the orthogonal geometry $V^{\frac{1}{2}}$, and we used Theorem \ref{CartanDieudonne} to write the outer automorphism as a product of reflections.  One can then calculate the spinor norm of the outer automorphisms.  We summarize our calculations below.

\begin{enumerate}
\item $A_{1}^{++}$.  There are no outer automorphism.
\item $A_{2}^{++}$.  There is a unique non-trivial outer automorphism $a$ and $\vartheta(a) = 3, \vartheta(-a) = -1$.
\item $A_{3}^{++}$.  There is a unique non-trivial outer automorphism $a$ and $\vartheta(a) = 2, \vartheta(-a) = -1$.
\item $A_{4}^{++}$.  There is a unique non-trivial outer automorphism $a$ and $\vartheta(a) = 5, \vartheta(-a) = -1$.
\item $A_{5}^{++}$.  There is a unique non-trivial outer automorphism $a$ and $\vartheta(a) = 3, \vartheta(-a) = -1$.
\item $A_{6}^{++}$.  There is a unique non-trivial outer automorphism $a$ and $\vartheta(a) = 7, \vartheta(-a) = -1$.
\item $A_{7}^{++}$.  There is a unique non-trivial outer automorphism $a$ and $\vartheta(a) = 1, \vartheta(-a) = -1$.
\item $D_{4}^{++}$.  There are $5$ non-trivial outer automorphisms corresponding to the following permutations of the simple roots:  $(1 \, 4), (1 \, 3), (3 \, 4), (1 \, 3 \, 4)$ and $(1 \, 4 \, 3)$.  We calculated $\vartheta((1 \, 4)) = \vartheta((1 \, 3)) = \vartheta((3 \, 4)) =2$.  Hence, $\vartheta((1 \, 3 \, 4)) = \vartheta((1 \, 4 \, 3)) = 1$.  Moreover, $\vartheta(-id) = -1$.
\item $D_{5}^{++}$.  There is a unique non-trivial outer automorphism $a$ and $\vartheta(a) = 2, \vartheta(-a) = -1$.
\item $D_{6}^{++}$.  There is a unique non-trivial outer automorphism $a$ and $\vartheta(a) = 2, \vartheta(-a) = -2$.
\item $D_{7}^{++}$.  There is a unique non-trivial outer automorphism $a$ and $\vartheta(a) = 2, \vartheta(-a) = -1$.
\item $D_{8}^{++}$.  There is a unique non-trivial outer automorphism $a$ and $\vartheta(a) = 2, \vartheta(-a) = -2$.
\item $E_{6}^{++}$.  There is a unique non-trivial outer automorphism $a$ and $\vartheta(a) = 3, \vartheta(-a) = -1$.
\item $E_{7}^{++}$ and $E_{8}^{++}$.  There are no outer automorphism.
\end{enumerate}

%\section{Connections with other description of the Weyl group}
\section{Connections with previous descriptions of the Weyl group} \label{connections}
In this section, we would like to explain the connection between our approach and some previous results contained in \cite{FeKlNi09}.

%\subsection{Paravectors}  
\subsection{Paravectors}  \label{paravector}
Throughout \S \ref{paravector}, we let $U$ be a non-singular orthogonal geometry and $\mathcal{C}$  a universal Clifford algebra for $U$.  We denote the corresponding quadratic form on $U$ by $q$.

\begin{Def}
The paravectors in $\mathcal{C}$ are defined to be the $F$-vector space $U_{para} = F \oplus U \subseteq \mathcal{C}$.
\end{Def} 
Note that if $x \in U_{para}$, then $x^{*} = x$ and $x' = \bar{x}$.  We define a symmetric $F$-bilinear form of $U_{para}$ via the formula
$$S_{para}(x,y) =\frac{1}{2}(x\bar{y}+y\bar{x}), $$
whenever $x,y \in U_{para}$.  The corresponding quadratic form will be denoted by $q_{para}$; hence, $q_{para}(x) = x \bar{x}$ for all $x \in U_{para}$ and one can check that $q_{para}(x) = x\bar{x} = \bar{x}x = q_{para}(\bar{x})$ whenever $x \in U_{para}$.  It is also simple to check the equality
\begin{equation} \label{usebelow}
S_{para}(x,\bar{y}) = S_{para}(\bar{x},y),
\end{equation}  
whenever $x,y \in U_{para}$.
If $x = \lambda_{1} + u_{1}$ and $y = \lambda_{2} + u_{2}$ for some $\lambda_{i} \in F$ and $u_{i} \in U$, then
$$S_{para}(x,y)= \lambda_{1} \lambda_{2} + S(u_{1},u_{2}). $$
It follows that if $u_{1},u_{2} \in U$, then $S_{para}(u_{1},u_{2}) = S(u_{1},u_{2})$, and if $\lambda_{1},\lambda_{2} \in F$, then $S_{para}(\lambda_{1},\lambda_{2}) = \lambda_{1}\lambda_{2}$.  Moreover, if $\lambda \in F$ and $u \in U$, then $S_{para}(\lambda,u) = 0$. Therefore, we have $U_{para} = F \perp U$.  Recall that $\mathcal{C}^{\times}$ acts on $\mathcal{C}$ via the Atiyah-Bott-Shapiro action.
\begin{Def}
Similarly as we did for $U$, we define a Clifford group as follows.  We let
$$\Gamma(U_{para}) = \{x \in \mathcal{C}^{\times} \, | \, x * U_{para} \subseteq U_{para} \}. $$
\end{Def}
Here is the relationship between $\Gamma(U)$ and $\Gamma(U_{para})$.
\begin{Pro}
With the notation as above, we always have $\Gamma(U) \subseteq \Gamma(U_{para})$.
\end{Pro}
\begin{proof}
Let $x \in \Gamma(U)$ and let $y = \lambda + u \in U_{para}$.  Then, we have $xy(x')^{-1} = x\lambda(x')^{-1} + xu(x')^{-1}$.  Since $x \in \Gamma(U)$, we have $xu(x')^{-1} \in U$ and we just have to show that $x\lambda (x')^{-1} \in F$, but this follows directly from Corollary \ref{clifford_generate}.
\end{proof}

From now on, we let $\rho_{para}$ denote the group morphism $\Gamma(U_{para}) \longrightarrow GL(U_{para})$ induced by the Atiyah-Bott-Shapiro action of $\Gamma(U_{para})$ on $U_{para}$.  In fact, $\rho_{para}(\Gamma(U_{para})) \subseteq O(U_{para})$.  Indeed, if $x \in \Gamma(U_{para})$ and $y \in U_{para}$, then
\begin{equation*}
q_{para}(xy(x')^{-1}) = xy(x')^{-1}(xy(x')^{-1})'= xyy'x^{-1} = q_{para}(y).
\end{equation*}
It follows that we have a group morphism $\rho_{para}:\Gamma(U_{para}) \longrightarrow O(U_{para})$.  If $x \in U_{para}$ is a non-isotropic vector, then $x \in \mathcal{C}^{\times}$ and its inverse is obviously given by
$$x^{-1} = \frac{\bar{x}}{q_{para}(x)}. $$
\begin{Pro} \label{imagenoniso}
For $x \in U_{para}$ non-isotropic and $y \in U_{para}$, we have $\rho_{para}(x)(y) = -r_{x}(\bar{y})$.
\end{Pro}
\begin{proof}
This is a simple calculation using (\ref{usebelow}) above and it goes as follows:
\begin{equation*}
x \cdot y \cdot (x')^{-1} = x \cdot y \cdot \frac{x}{q_{para}(x)} = \frac{1}{q_{para}(x)} x \cdot (- \bar{x} \cdot \bar{y} + 2 S_{para}(y,\bar{x})) = -r_{x}(\bar{y}).
\end{equation*}
\end{proof}
In particular, it follows from this last proposition that the non-isotropic vector $x \in U_{para}$ are in $\Gamma(U_{para})$.  Also, if $x \in U_{para}$ is non-isotropic, then $\rho_{para}(x) \in SO(U_{para})$.  We also have the following useful result:
\begin{Cor} \label{imageeven}
With the notation as above, if $x_{1}, x_{2} \in U_{para}$ are non-isotropic, then $\rho_{para}(x_{1} \cdot \bar{x}_{2}) = r_{x_{1}} \circ r_{x_{2}}$.
\end{Cor}
It follows from this last corollary and Theorem \ref{CartanDieudonne} that the group morphism
$$\rho_{para}:\Gamma(U_{para}) \longrightarrow SO(U_{para}) $$
is surjective.  Now, the exact same argument as in the proof of Proposition \ref{kernel} shows that ${\rm ker}(\rho_{para}) = F^{\times}$.  Thus, we have a short exact sequence

\begin{equation}
1 \longrightarrow F^{\times} \longrightarrow \Gamma(U_{para}) \stackrel{\rho_{para}}{\longrightarrow} SO(U_{para}) \longrightarrow 1.
\end{equation}

\begin{Cor} \label{generatorsforpara}
The group $\Gamma(U_{para})$ is generated by the non-isotropic vector $x \in U_{para}$.
\end{Cor}

We get a group morphism $N_{para}:\Gamma(U_{para}) \longrightarrow F^{\times}$ defined by $x \mapsto N_{para}(x) = x \bar{x}$.  We are now in the setting of section \S \ref{abstractspin}, and we define $Spin^{+}(U_{para}) = Spin^{+}(\rho_{para},N_{para})$.  As before, we have the exact sequence
$$1 \longrightarrow \{\pm 1\} \longrightarrow Spin^{+}(U_{para}) \stackrel{\rho_{para}}{\longrightarrow} SO(U_{para}) \stackrel{\vartheta}{\longrightarrow} F^{\times}/F^{\times 2}.$$
From now on, we let $SO^{+}(U_{para}) = \rho_{para}(Spin^{+}(U_{para}))$.

Following \cite{Po95}, we will now explain a connection between $\Gamma(U_{para})$ and $\Gamma^{0}(V)$ for a suitable orthogonal geometry $V$.  Let $L$ be a non-singular line with basis, say $e$.  Assume that the orthogonal geometry on $L$ is given by $S_{L}(\lambda_{1}e,\lambda_{2}e) = \lambda_{1}\lambda_{2}$ so that the associated quadratic form is given by $q_{L}(\lambda e) = \lambda^{2}$.  Then, consider $V = U \perp L$.  We clearly have an $F$-linear morphism $\xi:U \longrightarrow \mathcal{C}^{0}(V)$, given by $v \mapsto e v$.  Moreover, $(ev)^{2} = -q(v)$ and therefore, we get a morphism of $F$-algebras $\xi:\mathcal{C}(U) \longrightarrow \mathcal{C}^{0}(V)$ which we denote by the same symbol.  It is simple to check that if $x \in \mathcal{C}(U)$ is written as $x = x_{0} + x_{1}$ for some unique $x_{i} \in \mathcal{C}^{i}(U)$, then we have
\begin{equation} \label{fund}
\xi(x) = x_{0} + e x_{1}. 
\end{equation}
It follows that $\xi$ is an isomorphism of $F$-algebras.  Moreover, one also obtains from (\ref{fund}) the equality
\begin{equation} \label{fundeq}
e \cdot \xi(x') = \xi(x) \cdot e, 
\end{equation}
for all $x \in \mathcal{C}(U)$.  We have an obvious isometry $\sigma: U_{para} \longrightarrow V$ defined by $\lambda +u \mapsto \lambda e + u$.  A simple computation shows that
\begin{equation} \label{fundeq2}
e\cdot \xi(y') = \sigma(y) \text{ for }y \in U_{para}.
\end{equation}

\begin{Lem} \label{inordertocommute}
For $x \in \Gamma(U_{para})$ and $y \in U_{para}$ with $v = \sigma(y)$, one has $\xi(x) \cdot v \cdot \xi(x)^{-1} = \sigma(xy(x')^{-1})$.
\end{Lem}
\begin{proof}
Using (\ref{fundeq}) and (\ref{fundeq2}) above, we calculate
\begin{equation*}
\xi(x) \cdot v \cdot \xi(x)^{-1} =\xi(x)\cdot e \cdot\xi(y') \xi(x)^{-1} = e \cdot \xi(x'y'x^{-1}) = \sigma(xy(x')^{-1}).
\end{equation*}
\end{proof}
It follows from Lemma \ref{inordertocommute} that
$\xi(\Gamma(U_{para})) \subseteq \Gamma^{0}(V).$
If we let $\Sigma:SO(U_{para}) \stackrel{\simeq}{\longrightarrow} SO(V)$ be the isomorphism defined by $g \mapsto \sigma \cdot g \cdot \sigma^{-1}$, then it also follows from Lemma \ref{inordertocommute} that we have the commutative diagram
\begin{equation} \label{unyep}
  \begin{CD}
                 1 @>>> F^{\times} @>>> \Gamma(U_{para})       @>{\rho_{para}}>> SO(U_{para}) @>>>1 \\
                 & & @VV{\|}V  @VV{\xi}V       @VV{\Sigma}V   \\
                 1 @>>> F^{\times} @>>> \Gamma^{0}(V)       @>{\rho}>> SO(V) @>>> 1,
        \end{CD}
\end{equation}  
whose rows are exact and the vertical maps are all isomorphisms of groups.

%\subsection{Vahlen groups for paravectors}
\subsection{Vahlen groups for paravectors} \label{vahlenpara}
Let $U$ be a non-singular orthogonal geometry and let $\mathcal{C}(U)$ be a universal Clifford algebra for $U$.  As before, we let $W = U \perp P$, where $P$ is a hyperbolic plane.  We have the isomorphism of $F$-algebras $\phi:\mathcal{C}(W) \stackrel{\simeq}{\longrightarrow} M_{2}(\mathcal{C}(U))$ which leads us to the following definition.
\begin{Def}
With the notation as above, we define the following subgroup of $M_{2}(\mathcal{C}(U))^{\times}$:
$$\mathcal{V}(U_{para}) = \phi(\Gamma(W_{para})).$$
\end{Def}
We also let $H_{2}(U_{para}) = \phi(W_{para})$, that is
$$H_{2}(U_{para}) = \left\{X=\begin{pmatrix}x & \lambda_{1} \\ \lambda_{2} & \bar{x} \end{pmatrix} \middle| \, x \in U_{para} \text{ and } \lambda_{1}, \lambda_{2} \in F \right\} \subseteq M_{2}(\mathcal{C}(U)). $$
Then $\mathcal{V}(U_{para})$ acts on $H_{2}(U_{para})$ via $A \cdot X = A X A^{\sharp}$.  We get a representation 
$$\eta: \mathcal{V}(U_{para}) \longrightarrow SO(H_{2}(U_{para})), $$
which fits into the following commutative diagram:
\begin{equation} \label{deuxyep}
  \begin{CD}
                 1    @>>>  F^{\times}      @>>>   \Gamma(W_{para})       @>{\rho_{para}}>> SO(W_{para}) @>>> 1  \\
                 & & @|   @VV{\phi}V       @VV{\Phi}V   \\
                 1    @>>>   F^{\times}      @>>>        \mathcal{V}(U_{para})       @>{\eta}>> SO(H_{2}(U_{para})) @>>> 1 ,
        \end{CD}
\end{equation}  
whose vertical arrows are isomorphisms of groups.  For the exact same reason as in the non-paravector situation, we have the following more concrete characterization of Vahlen groups.
\begin{The} \label{vahlenpara1}
Let 
\begin{equation*}
A =
\begin{pmatrix}
a & b \\
c & d
\end{pmatrix} \in M_{2}(\mathcal{C}(U)).
\end{equation*}
Then $A \in \mathcal{V}(U_{para})$ if and only if the following conditions are satisfied.
\begin{enumerate}
\item $ad^{*} - bc^{*} = d^{*}a - b^{*}c = \lambda\in F^{\times}$, \label{un1}
\item $ba^{*} - ab^{*} = cd^{*} - dc^{*} = 0$, \label{deux1}
\item $a^{*}c - c^{*}a = d^{*}b - b^{*}d = 0$, \label{trois1}
\item $a \bar{a}, b \bar{b}, c \bar{c}, d \bar{d} \in F$, \label{quatre1}
\item $b \bar{d}, a \bar{c} \in U_{para}$, \label{cinq1}
\item $ax\bar{b} + b\bar{x}\bar{a}, cx\bar{d} + d\bar{x}\bar{c} \in F$ for all $x \in U_{para}$, \label{six1}
\item $ax\bar{d} + b\bar{x} \bar{c} \in U_{para}$ for all $x \in U_{para}$. \label{sept1}
\end{enumerate}
\end{The}

We also have in this context a group morphism $N:\mathcal{V}(U_{para}) \longrightarrow F^{\times}$ given by $A \mapsto N(A) = A \cdot \gamma(A)$.  This puts us in the context of \S \ref{abstractspin}, and we let
$$S\mathcal{V}^{+}(U_{para}) = Spin^{+}(\eta,N).$$
More concretely, we have the following result:
\begin{The}
One has $A \in S\mathcal{V}^{+}(U_{para})$ if and only if (\ref{deux1}) through (\ref{sept1}) of Theorem \ref{vahlenpara1} are satisfied and (\ref{un1}) is replaced by $ad^{*} - bc^{*} = d^{*}a - b^{*}c = 1 $.
\end{The}

The reader should compare our group $S\mathcal{V}^{+}(U_{para})$ to the group ${\rm SV}_{n}(F,q)$ of Definition $3.1$ of \cite{ElGrMe87}.  The analogous result to Proposition \ref{imagenoniso} is given by the following proposition:
\begin{Pro} \label{formula_for_reflections_para}
Let $X,Y \in H_{2}(U_{para})$ with $X$ non-isotropic.  Then $\eta(X)(Y) = -r_{X}(\gamma(Y))$.
\end{Pro}
\begin{proof}
The proof is exactly the same as the one for Proposition \ref{imagenoniso}.
\end{proof}
As a corollary, we obtain the following result.
\begin{Cor} \label{formulaforgen}
Given $X_{1},X_{2} \in H_{2}(U_{para})$, both non-isotropic, then $\eta(X_{1} \cdot \gamma(X_{2})) = r_{X_{1}} \circ r_{X_{2}}$.
\end{Cor}

Combining (\ref{unyep}) and (\ref{deuxyep}), we obtain the following commutative diagram 
\begin{equation} 
  \begin{CD}
                 1    @>>>  F^{\times}      @>>>   \mathcal{V}(U_{para})       @>{\eta}>> SO(H_{2}(U_{para})) @>>> 1  \\
                 & & @|   @VV{\Xi}V       @VV{\Sigma}V   \\
                 1    @>>>   F^{\times}      @>>>        \mathcal{V}^{0}(V)       @>{\eta}>> SO(H_{2}(V) @>>> 1 ,
        \end{CD}
\end{equation}  
whose vertical arrows are all isomorphisms of groups and where $V = U \perp L$ and $L =Fe$ is a line whose quadratic form is given by $q(e) = 1$.  The isometry $\sigma:U_{para} \longrightarrow V$ induces another obvious isometry $H_{2}(U_{para}) \longrightarrow H_{2}(V)$ which in turns induces the group isomorphism $\Sigma$ of this last commutative diagram.  The group isomorphism $\Xi$ comes from the group isomorphism $\xi$ of (\ref{fund}) and is explicitly given by the following formula.  If 
\begin{equation*}
A = 
\begin{pmatrix}
a & b \\
c & d
\end{pmatrix}
\in \mathcal{V}(U_{para})
\qquad
\hbox{then}
\qquad
\Xi(A)=
\begin{pmatrix}
g_{1}a + g_{2}a' & g_{1}b - g_{2}b' \\
g_{2}c - g_{1}c' & g_{2}d + g_{1}d'
\end{pmatrix},
\end{equation*}
where
$$g_{1} = \frac{1 + e}{2} \text{ and } g_{2} = \frac{1-e}{2}. $$
Note that $g_{i} \in \mathcal{C}(V)$ and $g_{1}' = g_{2}$ and $g_{2}' = g_{1}$.

One difference between our approach and the one adopted in \cite{FeKlNi09} is that we work in a non-paravector situation whereas they work in a paravector situation.  We shall explain this further for $A_{1}^{++}$ and $A_{2}^{++}$ below.

%\subsection{A_{1}^{++}}
\subsubsection{$A_{1}^{++}$}
In this section, we let $F = \mathbb{Q}$ and we consider $V = \mathbb{Q}$ as an orthogonal geometry as we did in \S \ref{symmetric} for the Cartan matrix of finite type $C = (2)$ corresponding to $A_{1}$.  Note that $V^{\frac{1}{2}}$ is isometric to the line $L  =\mathbb{Q} e$, whose quadratic form is given by $q(e) = 1$.  Letting $U = \{0 \}$ with the trivial quadratic form, we have $V^{\frac{1}{2}} = U \perp L$.  Note that $\mathcal{C}(U)$ is isomorphic to $\mathbb{Q}$ as $\mathbb{Q}$-algebras and $U_{para} = \mathcal{C}(U) = \mathbb{Q}$.  It is simple to check that
\begin{enumerate}
\item $\mathcal{V}(U_{para}) = GL(2,\mathbb{Q})$,
\item $S\mathcal{V}^{+}(U_{para}) = SL(2,\mathbb{Q})$.
\end{enumerate}
The $\mathbb{Q}$-algebra $\mathcal{C}(V^{\frac{1}{2}})$ is isomorphic to $\mathbb{Q}(i)$ and the embedding $V^{\frac{1}{2}} = L \hookrightarrow \mathcal{C}(V^{\frac{1}{2}}) = \mathbb{Q}(i)$ is given by $xe \mapsto xi$. Moreover, one can check that
\begin{equation*}
\mathcal{V}^{0}(V^{\frac{1}{2}}) = \left \{\begin{pmatrix} a & b \\ c & d \end{pmatrix} \in GL(2,\mathbb{Q}(i)) \, \Big|\,  a,d \in \mathbb{Q} \text{ and } b,c \in \mathbb{Q}i \right \}.
\end{equation*}
The group isomorphism $\Xi:GL(2,\mathbb{Q}) \longrightarrow \mathcal{V}^{0}(V^{\frac{1}{2}})$ of the last section is given explicitly by
\begin{equation*}
\Xi \left(\begin{pmatrix} a & b \\ c & d \end{pmatrix} \right) = \begin{pmatrix} a & ib \\ -ic & d\end{pmatrix}.
\end{equation*}
Note also that the order $\mathcal{O}$ in $\mathcal{C}(V^{\frac{1}{2}}) = \mathbb{Q}(i)$ is nothing else than $\mathbb{Z}[i]$.  Therefore,
\begin{equation*}
\mathcal{SV}^{+}(\mathcal{O}) = \left \{\begin{pmatrix} a & ib \\ -ic & d \end{pmatrix} \in SL(2,\mathbb{Q}(i)) \, \Big|\,  a,b,c,d \in \mathbb{Z} \right \},
\end{equation*}
and we clearly have $\Xi\left(SL(2,\mathbb{Z}) \right) = \mathcal{SV}^{+}(\mathcal{O})$.

In summary, there are two possibilities in order to study the even part of the Weyl group of the hyperbolic Kac-Moody algebra $A_{1}^{++}$.  One can work with the short exact sequence
$$1 \longrightarrow \mathbb{Q}^{\times} \longrightarrow \mathcal{V}^{0}(V^{\frac{1}{2}}) \stackrel{\eta}{\longrightarrow} SO(H_{2}(V^{\frac{1}{2}})) \longrightarrow 1,$$
as we did or one can work with the short exact sequence
$$1 \longrightarrow \mathbb{Q}^{\times} \longrightarrow GL(2,\mathbb{Q}) \stackrel{\eta}{\longrightarrow} SO(H_{2}(\mathbb{Q})) \longrightarrow 1,$$
where
\begin{equation*}
H_{2}(\mathbb{Q}) = \left \{ \begin{pmatrix} x & \lambda_{1} \\ \lambda_{2} & x \end{pmatrix} \, \Big| \, \lambda_{1},\lambda_{2},x \in \mathbb{Q} \right \}.
\end{equation*}
In \cite{FeKlNi09}, they use the second approach.  Moreover, they are not working with $H_{2}(\mathbb{Q})$, but rather with the two by two symmetric matrices with rational coefficients.  We will explain this other option in \S \ref{parahermitian}.

%\subsection{$A_{2}^{++}$}
\subsubsection{$A_{2}^{++}$}
In this section, we let $F = \mathbb{Q}$ and we consider $V = \mathbb{Q}^{2}$ as an orthogonal geometry as we did in \S \ref{symmetric} for the Cartan matrix of finite type 
\begin{equation*}
C = 
\begin{pmatrix}
2 & -1 \\ -1 & 2
\end{pmatrix}
\end{equation*}
corresponding to $A_{2}$.  Note that an orthogonal basis for $V^{\frac{1}{2}}$ is given by $(\beta_{1},\beta_{2})$, where $\beta_{1} = \alpha_{1}$ and $\beta_{2} = \alpha_{1} + 2\alpha_{2}$.  Note also that 
$$\frac{1}{2}\kappa(\beta_{2},\beta_{2}) = 3. $$
This leads us to consider $U = \mathbb{Q}$ with the quadratic form given by $q(x) = 3x^{2}$.  Then, we have $V^{\frac{1}{2}} = \mathbb{Q}\alpha_{1} \perp U$.  Then, the Clifford algebra $\mathcal{C}(U)$ is isomorphic as a $\mathbb{Q}$-algebra to $\mathbb{Q}(\sqrt{-3})$.  The embedding $\mathbb{Q} \hookrightarrow \mathbb{Q}(\sqrt{-3})$ is given be $x \mapsto x\sqrt{-3}$.  Moreover $U_{para} = \mathbb{Q}(\sqrt{-3})$.  It is simple to check that
\begin{enumerate}
\item $\mathcal{V}(U_{para}) = GL(2,\mathbb{Q}(\sqrt{-3})) \cap {\rm det}^{-1}(\mathbb{Q}^{\times})$,
\item $S\mathcal{V}^{+}(U_{para}) = SL(2,\mathbb{Q}(\sqrt{-3}))$.
\end{enumerate}
Now, the $\mathbb{Q}$-algebra $\mathcal{C}(V^{\frac{1}{2}})$ is isomorphic as a $\mathbb{Q}$-algebra to the quaternion algebra $\left(\frac{-1\,,\,-3}{\mathbb{Q}} \right)$, the embedding is given by $\beta_{1} \mapsto i$ and $\beta_{2} \mapsto j$ so that $i^{2}=-1$ and $j^{2} = -3$.  If $w + xi + yj +zk \in \left(\frac{-1\,,\,-3}{\mathbb{Q}} \right)$, then the (anti)-involutions are given by
$$(w + xi + yj +zk)^{*} = w + xi + yj -zk, \, (w + xi + yj +zk)^{'} = w - xi - yj +zk,$$
and
$$\overline{w + xi + yj +zk} = w - xi - yj -zk.$$
The even and odd parts of the Clifford algebra are given by
$$\left(\frac{-1 \, , \, -3}{\mathbb{Q}} \right)^{0} = \left\{w + zk \, \Big| \, w,z \in \mathbb{Q} \right\}\qquad\hbox{and}\qquad\left(\frac{-1 \, , \, -3}{\mathbb{Q}} \right)^{1} = \left\{xi + yj \, \Big| \, x,y \in \mathbb{Q} \right\}.$$
The Clifford algebra $\mathcal{C}(U)$ embeds in the Clifford algebra $\mathcal{C}(V^{\frac{1}{2}})$ via $\sqrt{-3} \mapsto j$ and therefore, we will identify $\sqrt{-3}$ with $j$.  One can check that
\begin{equation*}
\mathcal{V}^{0}(V^{\frac{1}{2}}) = \left\{\begin{pmatrix} a & b \\ c & d\end{pmatrix} \in M_{2}\left(\left(\frac{-1 \, , \, -3}{\mathbb{Q}} \right) \right)\, \Big| \, a,d \in \left(\frac{-1 \, , \, -3}{\mathbb{Q}} \right)^{0}, b,c \in \left(\frac{-1 \, , \, -3}{\mathbb{Q}} \right)^{1}, ad^{*} - bc \in \mathbb{Q}^{\times} \right\}.
\end{equation*}
Moreover, the group isomorphism $\Xi: \mathcal{V}(U_{para}) \longrightarrow \mathcal{V}^{0}(V^{\frac{1}{2}})$ is explicitly given by
\begin{equation*}
\Xi\left(\begin{pmatrix}x_{1} + y_{1}j & x_{2} + y_{2}j \\ x_{3} + y_{3}j & x_{4} + y_{4}j \end{pmatrix} \right) = \begin{pmatrix}x_{1} + y_{1}k & x_{2}i + y_{2}j \\ -x_{3}i + y_{3}j & x_{4} - y_{4}k \end{pmatrix}.
\end{equation*}
One can also check that the order $\mathcal{O} =\mathbb{Z}[\alpha_{1},\alpha_{2}] \subseteq \mathcal{C}(V^{\frac{1}{2}}) = \left(\frac{-1 \, , \, -3}{\mathbb{Q}} \right)$ is given by
$$\mathcal{O} = \left\langle 1,i,\frac{j-i}{2},\frac{k+1}{2} \right\rangle, $$
and that $\Xi(SL(2,O_{-3})) = S\mathcal{V}^{+}(\mathcal{O})$, where $O_{-3}$ is the ring of integers of the number fields $\mathbb{Q}(j)$.

So again, there are two possibilities in order to study the even part of the Weyl group of the hyperbolic Kac-Moody algebra $A_{2}^{++}.$  One can work with the short exact sequence
$$1 \longrightarrow \mathbb{Q}^{\times} \longrightarrow \mathcal{V}^{0}(V^{\frac{1}{2}}) \stackrel{\eta}{\longrightarrow} SO(H_{2}(V^{\frac{1}{2}})) \longrightarrow 1,$$
as we did or one can work with the short exact sequence
$$1 \longrightarrow \mathbb{Q}^{\times} \longrightarrow GL(2,\mathbb{Q}(\sqrt{-3}))\cap {\rm det}^{-1}(\mathbb{Q}^{\times}) \stackrel{\eta}{\longrightarrow} SO(H_{2}(\mathbb{Q}(\sqrt{-3}))) \longrightarrow 1,$$
where
\begin{equation*}
H_{2}(\mathbb{Q}(\sqrt{-3})) = \left \{ \begin{pmatrix} x & \lambda_{1} \\ \lambda_{2} & x^{j} \end{pmatrix} \, \Big| \, \lambda_{1},\lambda_{2} \in \mathbb{Q}, x \in \mathbb{Q}(\sqrt{-3}) \right \}.
\end{equation*}
In this last equation, $j$ is the unique non-trivial Galois automorphism of $\mathbb{Q}(\sqrt{-3})/\mathbb{Q}$.  In \cite{FeKlNi09}, they use the second approach.  Moreover, they are not working with $H_{2}(\mathbb{Q}(\sqrt{-3}))$, but rather with the two by two Hermitian matrices with coefficients in $\mathbb{Q}(\sqrt{-3})$.  We explain this below in \S \ref{parahermitian}.

%\subsection{Hermitian matrices}
\subsection{Hermitian matrices} \label{parahermitian}
Let $U$ be a non-singular orthogonal geometry.  It has been more customary to work with Hermitian matrices rather than $H_{2}(U_{para})$.  We now explain the connection between the two.  We explain this in the context of paravectors, but a similar phenomenon is happening in the non-paravector situation as well.  

We let
$$\widetilde{H}_{2}(U_{para}) = \left\{X =\begin{pmatrix} \lambda_{1} & x \\ \bar{x} & \lambda_{2} \end{pmatrix} \middle| \, x \in U_{para} \text{ and } \lambda_{1}, \lambda_{2} \in F \right\}\subseteq M_{2}(\mathcal{C}(U)). $$
We have an obvious isomorphism of $F$-vector spaces $\psi:H_{2}(U_{para}) \longrightarrow \widetilde{H}_{2}(U_{para})$ given by
\begin{equation*}
\begin{pmatrix}
x & \lambda_{1}\\ \lambda_{2} & \bar{x}
\end{pmatrix}
\mapsto 
\psi\left( \begin{pmatrix}
x & \lambda_{1}\\ \lambda_{2} & \bar{x}
\end{pmatrix} \right)=
\begin{pmatrix}
\lambda_{1} & x\\  \bar{x} & \lambda_{2}
\end{pmatrix}
\end{equation*}
Hence, the $F$-vector space $\widetilde{H}_{2}(U_{para})$ inherits an orthogonal geometry from $H_{2}(U_{para})$ whose corresponding quadratic form is given by $Q(X) = x\bar{x}-\lambda_{1}\lambda_{2}$.  The corresponding symmetric $F$-bilinear form will be denoted by $S$.  From now on, we let
\begin{equation*}
E_{2} =
\begin{pmatrix}
0 & 1\\
1 & 0
\end{pmatrix}.
\end{equation*}
Note that if $X \in H_{2}(U_{para})$ then $\psi(X) = XE_{2}$.  It is simple to check that $E_{2} \in \mathcal{V}(U_{para})$. 

Given
\begin{equation*}
A =
\begin{pmatrix}
a & b \\
c & d
\end{pmatrix} \in \mathcal{V}(U_{para})
\qquad
\hbox{we let}
\qquad
A^{\dagger} = \frac{1}{\lambda} 
\begin{pmatrix}
\bar{a} & \bar{c}\\
\bar{b} & \bar{d}
\end{pmatrix},
\qquad \text{where } \lambda = ad^{*} - bc^{*}.  
\end{equation*}
Note that if $A \in \mathcal{V}(U_{para})$, then $A^{\dagger} = E_{2} A^{\sharp} E_{2}$.  Since $\mathcal{V}(U_{para})$ acts on $H_{2}(U_{para})$, it also acts on $\widetilde{H}_{2}(U_{para})$ via $\psi$.  It is a simple matter to check that the action is given by $A \cdot X = AXA^{\dagger}$ whenever $A \in \mathcal{V}(U_{para})$ and $X \in \widetilde{H}_{2}(U_{para})$.  Hence, we get a representation 
$$\tilde{\eta}:\mathcal{V}(U_{para}) \longrightarrow SO(\widetilde{H}_{2}(U_{para})).$$
Moreover, the isometry $\psi$ induces an isomorphism of groups $\Psi:SO(H_{2}(U_{para})) \longrightarrow SO(\widetilde{H}_{2}(U_{para}))$ given by $\sigma \mapsto \Psi(\sigma) = \psi \circ \sigma \circ \psi^{-1}$.  We then obtain the following commutative diagram:
\begin{equation} \label{shortexact3}
  \begin{CD}
                 1    @>>>  F^{\times}      @>>>   \mathcal{V}(U_{para})       @>{\eta}>> SO(H_{2}(U_{para})) @>>> 1  \\
                 & & @|   @|       @VV{\Psi}V   \\
                 1    @>>>   F^{\times}      @>>>        \mathcal{V}(U_{para})       @>{\tilde{\eta}}>> SO(\widetilde{H}_{2}(U_{para})) @>>> 1 ,
        \end{CD}
\end{equation}  
whose rows are exact.

The reader will have no difficulty in checking that Proposition \ref{formula_for_reflections_para} is valid here as well.
\begin{Pro} \label{imp1}
Let $X,Y \in \widetilde{H}_{2}(U_{para})$ with $X$ non-isotropic, then $\tilde{\eta}(X)(Y) = - r_{X}(\gamma(Y))$.
\end{Pro}

\begin{Cor} \label{imp2}
Given $X_{1},X_{2} \in \widetilde{H}_{2}(U_{para})$, both non-isotropic, then $\tilde{\eta}(X_{1} \cdot \gamma(X_{2})) = r_{X_{1}} \circ r_{X_{2}}$.
\end{Cor}
We also have the group morphism $N:\mathcal{V}(U_{para}) \longrightarrow F^{\times}$ defined as before by $A \mapsto A \cdot \gamma(A)$.  Then, we obtain the important exact sequence
\begin{equation} \label{paraspinplusexact}
1 \longrightarrow \{\pm 1 \} \longrightarrow S\mathcal{V}^{+}(U_{para}) \stackrel{\tilde{\eta}}{\longrightarrow} SO(\widetilde{H}_{2}(U_{para})) \stackrel{\vartheta}{\longrightarrow} F^{\times}/F^{\times 2}.
\end{equation}

Note that $\vartheta$ satisfies $\vartheta(r_{X_{1}} \circ r_{X_{2}}) = Q(X_{1})Q(X_{2}) \cdot F^{\times 2}$ whenever $X_{1},X_{2} \in \widetilde{H}_{2}(U_{para})$ are non-isotropic vectors.

One should compare the above with \S $3.1$ and \S $3.2$ of \cite{FeKlNi09}.  There, a different formula than the one of Proposition \ref{imp1} is presented.  One can check that the matrices coming from Corollary \ref{imp2} are the same up to a sign as the ones of Theorem $2$ of \cite{FeKlNi09}.  Our treatment here is simplified a little, since we are using the anti-involution $\gamma$ instead of $X \mapsto \bar{X}$ which is neither an involution nor an anti-involution in general, unless we are in a commutative setting.

In summary, if $V$ is one of the orthogonal geometries of \S \ref{symmetric}, and one writes $V^{\frac{1}{2}} = L \perp U$ for some orthogonal geometry $U$, where $L = \mathbb{Q}e$ is a line with quadratic form $q(e) = 1$, then the authors of \cite{FeKlNi09} are working with the short exact sequence
$$1 \longrightarrow \mathbb{Q}^{\times} \longrightarrow \mathcal{V}(U_{para}) \stackrel{\tilde{\eta}}{\longrightarrow} SO(\widetilde{H}_{2}(U_{para})) \longrightarrow 1,$$
whereas we are rather working with the following short exact sequence
$$1 \longrightarrow \mathbb{Q}^{\times} \longrightarrow \mathcal{V}^{0}(V^{\frac{1}{2}}) \stackrel{\eta}{\longrightarrow} SO(H_{2}(V^{\frac{1}{2}})) \longrightarrow 1.$$
Our approach has the advantage that we can study directly the Weyl group instead of only the even part of the Weyl group.  Moreover, the formulas giving the reflections are simpler in our situation.  To be precise, the authors of \cite{FeKlNi09} also work over $\mathbb{R}$ instead of $\mathbb{Q}$, but it seems to clarify the picture quite a bit if one works over $\mathbb{Q}$ instead of over $\mathbb{R}$.  As a last remark, we also point out that we are not dealing with an explicit set of known generators for the groups $\mathcal{V}(V^{\frac{1}{2}})$ as was done in \cite{FeKlNi09} for $SL(2,\mathbb{Z})$ and $SL(2,O_{-3})$ for instance.  We are replacing this argument with the use of a corollary of Kac as we explained in the proof of Theorem \ref{main}.

\end{document}